\documentclass[11pt]{elsarticle}
\usepackage{mathrsfs, amsthm, amsmath, amssymb, anysize}

\newtheorem{lemma}{Lemma}[section]
\newtheorem{theorem}[lemma]{Theorem}
\newtheorem{remark}[lemma]{Remark}
\newtheorem{proposition}[lemma]{Proposition}

\newtheorem{corollary}[lemma]{Corollary}

\newtheorem{convention}[lemma]{Convention}
\marginsize{3.6cm}{3.6cm}{2.4cm}{2cm}%×óÓÒÉÏÏÂ
\biboptions{}\journal{Journal of Algebra and its Applications.}
\begin{document}
\begin{frontmatter}

\title{Maximal  Subalgebras for   Lie Superalgebras of Cartan Type}
\author{Wei Bai\fnref{fn1}}
\address{Department of Mathematics, Harbin Institute of Technology, \\ Harbin 150006,  P. R. China}
\address{School of Mathematical Sciences, Harbin Normal University, \\ Harbin 150025, P. R. China}
\author{Wende Liu\fnref{fn2}*}
\address{Department of Mathematics, Harbin Institute of Technology, \\ Harbin 150006,  P. R. China}
\author{Xuan Liu}
\address{Applied Mathematics Department, The University of Western Ontario,\\ London, N6A 5B7, Canada}

\author{Hayk Melikyan\fnref{fn3}}
%\ead{melikyan@nccu.edu}
\address{Department of Mathematics and Computer Science, North Carolina Central University\\ Durham, NC 27713, USA}

\fntext[fn1]{Supported by  NSF of the Education Department of HLJP (12521158)}
\fntext[fn2]{Supported by NSF of China (11171055) and NSF of HLJP (JC201004, A2010-03)}
\fntext[fn3]{Supported by part NSF Grant \# 0833184}
\cortext[cor1]{Corresponding author: wendeliu@ustc.edu.cn (W. Liu)}
\markright{Maximal  Subalgebras for   Lie Superalgebras of Cartan Type}

\begin{abstract}
 The maximal graded subalgebras for four families of Lie superalgebras of Cartan type over a field of prime characteristic are studied. All  maximal reducible graded subalgebras are described completely and their isomorphism classes,   dimension formulas are found.
 The classification of  maximal irreducible graded subalgebras is reduced to the classification of the maximal irreducible subalgebras for the classical   Lie superalgebras  $\frak{gl}(m, n)$, $\frak{sl}(m, n)$ and $\frak{osp}(m, n)$.\\
%\section{}
\end{abstract}
\end{frontmatter}

\setcounter{section}{-1}
\section{Introduction}

Since V. G. Kac \cite{Kac1} classified the finite dimensional simple Lie superalgebras over algebraically
closed fields  of characteristic zero, the theory of Lie superalgebras has undergone  a significant development (for example \cite{Sch, Kac2}).
  %Excellent systematic studies were performed by M. Scheunent \cite{Sch}, Y. Bahturin and his students \cite{BMPZ}.
Over a field of finite characteristic, however, the classification problem is still open for the finite dimensional simple Lie superalgebras \cite{Zh, bl}. Even recently, new simple  Lie superalgebras over a field of characteristic $ p = 3$ were constructed \cite{bl,e1}.

In general, study of the maximal subsystems of an algebraic system, such as finite groups, Lie groups, Lie  (super)algebras, is an essential part of structural characterization of the system. In classical Lie theory, the classification of maximal subalgebras of simple Lie algebras over the field of complex numbers is one of the beautiful results of that theory which was due to E. Dynkin \cite {Dyn1, Dyn2}.
%Moreover, because over the field of complex numbers results for Lie algebras are essentially interchangeable with results for the corresponding simply connected Lie groups, the analogous results for groups followed afterwards
 In classical modular Lie theory there is a series of papers by G. Seitz and his students devoted to the study of the maximal subgroups of simple algebraic groups over fields of  positive characteristic. These investigations were
summarized by G. Seitz in his two publications \cite{Se1,Se2} which generalize E. Dynkin's
 classification of the maximal subgroups of simple Lie groups over the field of complex numbers
 \cite{Dyn2} to simple algebraic groups over fields of characteristic $ p > 7$.
The study of maximal subalgebras of different classes of (super)algebras has been the focus of several researchers.
 The maximal subalgebras of Jordan (super)algebras were studied by M. Racine \cite{Re1, Re2}, A. Elduque, J. Laliena and
 S. Sacristan \cite{ELS1, ELS2}.
 %I. Shchepochkina in \cite{Sh2, Sh3} studied  maximal subalgebras of finite dimensional  simple matrix superalgebras over the complex numbers.
 The maximal graded subalgebras of affine Kac-Moody algebras were classified in \cite{BS}.
   The fourth author of the present paper summarized his investigations on maximal subalgebras in Cartan type simple Lie algebras over the field of characteristic $p >3$ in his  paper \cite{HM}.

Let  $L$ be a  finite dimensional
 simple Lie superalgebras of Cartan type $W$, $S$, $H$ or $K$ with a $\mathbb{Z}$-grading  $L=\oplus_{i\geq -2}L_i$.
The present paper is devoted to characterizing the maximal graded subalgebras of $L$.
 To this end,  we construct  a series of    graded subalgebras of $L$ and state the necessary and sufficient conditions for their maximality. Moreover,
  the number of isomorphism classes  and the dimension formulas of all maximal graded subalgebras are completely determined except for maximal irreducible graded subalgebras.
  Note that the null of $L$  is isomorphic to a classical  Lie superalgebra (see Lemma \ref{dibu}(3)).
Thus the classification of the maximal irreducible graded subalgebras of $L$ is reduced to that of the maximal irreducible subalgebras of a  classical  Lie superalgebra.
Moreover,  we give necessary and sufficient conditions for the existence of maximal irreducible graded  subalgebras of $L$.
%The authors expect that the work in this aspect will be helpful to further understand the intrinsic properties of the modular Lie superalgebras of Cartan type.
We should mention that the present work which partially generalizes the results of \cite{HM} is motivated by a paper by A. I. Kostrikin and I. R. Shafarevich \cite{KS} on the structure theory of modular Lie algebras.

We close this introduction by establishing the following  conventions: The underlying field $\mathbb{F}$ is an algebraically closed field of characteristic $p > 3$. In addition to the standard notation $\mathbb{Z},$ we write $\mathbb{N}$ for
the set of nonnegative  integers. The field of two elements is denoted by $\mathbb{Z}_{2}=\{\bar{0}, \bar{1}\}.$
For a proposition $P$, put
$\delta_{P}=1 $ if $ P$  is true and $\delta_{P}=0$ otherwise.
\textit{All subspaces, subalgebras and submodules are assumed
to be $\mathbb{Z}_{2}$-graded and all the homomorphisms of  $\mathbb{Z}$-graded superalgebras are both $\mathbb{Z}_{2}$-homogeneous and $\mathbb{Z}$-homogeneous.}

\section{Basics}
Fix two positive integers $m, n\in\mathbb{N}\backslash\{1\}.$ Put
 $$\mathbf{I}_0=\overline{1,m},\ \mathbf{I}_{1}=\overline{m+1,m+n},\
\mathbf{I}=\mathbf{I}_0\cup\mathbf{I}_1,$$
where $\overline{k, s}=\{k, k+1, \ldots, s\}$
with the convention $\overline{k, s}=\emptyset$ whenever $k>s.$
Write
$$\mathbf{A}(m)=\{\alpha=(\alpha_{1},\ldots,\alpha_{m})\in\mathbb{N}^{m}\mid
0\leq\alpha_{i}\leq p-1,i\in\mathbf{I}_{0} \}.$$
Let $\mathcal {O}(m)$ be
the \textit{divided power algebra} with
$\mathbb{F}$-basis
 $\{x^{(\alpha)}\mid\alpha\in\mathbf{A}(m)\}$ and
 $\Lambda(n)$  be the \textit{exterior
superalgebra} of $n$ variables $
x_{m+1},x_{m+2},\ldots,x_{m+n}$. The tensor product
$$\mathcal {O}(m,n)=\mathcal {O}(m)\otimes\Lambda(n)$$
 is an associative superalgebra with respect to the usual
$\mathbb{Z}_{2}$-grading.
Let
$$\mathbf{B}(n)=\{\langle i_{1},i_{2},\ldots,i_{k}\rangle\mid 0\leq  k\leq  n;
m+1\leq i_{1}<i_{2}<\cdots<i_{k}\leq m+n\}$$
be the set of $k$-tuples of strictly increasing integers in $\mathbf{I}_1$, $0\leq  k\leq  n.$
 For $u=\langle
i_1,i_2,\ldots,i_k\rangle\in \mathbf{B}(n)$,  write $x^{u}=x_{i_{1}}x_{i_{2}}\cdots x_{i_{k}}$ ($x^{\emptyset}=1$).
If $g\in\mathcal {O}(m)$ and $f\in \Lambda(n)$,  we  write $gf$ instead of $g\otimes
f$. Then $\mathcal {O}(m,n)$ has a $\mathbb{Z}_{2}$-homogeneous $\mathbb{F}$-basis
$$\{x^{(\alpha)}x^{u}\mid \alpha\in \mathbf{A}(m),u\in
\mathbf{B}(n)\}.$$
For $i\in \textbf{I}_{0}$ and
$\varepsilon_{i}=(\delta_{i1},\delta_{i2},\ldots,\delta_{im})$, write $x_{i}$ for
  $x^{(\varepsilon_{i})}.$ Let $\partial_{1},\ldots,\partial_{m+n}$ be the \textit{superderivations}
of the superalgebra $\mathcal {O}(m,n)$ such that $\partial_{i}(x_j)=\delta_{i=j}.$
The \textit{parity} of $\partial_{i}$ is
$|\partial_{i}|= \bar{0}$ if $i\in \mathbf{I}_{0}$ and
$\bar{1}$ if $i\in \mathbf{I}_{1}$.
\textit{Hereafter the symbol $|x|$ implies that $x$ is a $\mathbb{Z}_{2}$-homogeneous element.}
Put
$$
W(m,n) =\mathrm{span}_{\mathbb{F}}\left\{a\partial_i \mid a\in \mathcal
{O} ( m,n), i\in \mathbf{I}\right\},
$$
which is a finite dimensional simple Lie superalgebra, called \textit{Witt superalgebra}.
Consider the linear mapping called divergence:
$$\mathrm{div}:W(m,n)\longrightarrow\mathcal {O}(m,n),\quad
\mathrm{div}(f\partial_{i})=(-1)^{|f||\partial_i|}\partial_{i}(f).$$
Set  $S(m,n)=[\overline{S}(m,n),\overline{S}(m,n)]$,
where $\overline{S}(m,n)=\ker(\mathrm{div})$.
Then we have
$$
S(m,n)=\mathrm{span}_{\mathbb{F}}\left\{
D_{ij}(a)\mid a\in \mathcal {O}(m,n),i,j\in\mathbf{I}\right\},
$$
where
$$
D_{ij}(a)=(-1)^{|\partial_i||\partial_j|}\partial_{i}(a)\partial_{j}-(-1)^{(|\partial_i|+|\partial_j|)|a|}
\partial_{j}(a)\partial_{i}\quad\mbox{for } a\in\mathcal{O}(m,n).
$$
 $S(m,n)$ is a  simple
Lie superalgebra, called \textit{special superalgebra}.

For $j\in\{1,\ldots,2\lfloor\frac{m}2\rfloor, m+1,\ldots, m+n\}$, we put
\[
\sigma(j)=\left\{\begin{array}{ll}
 -1, & j\in\overline{\lfloor\frac{m}2\rfloor+1,2\lfloor\frac{m}2\rfloor}; \\
 1, & \mbox { otherwise }
                 \end{array}
\right.
\mbox{ and }\
j'=\left\{\begin{array}{lll}
 j+\lfloor\frac{m}2\rfloor ,  & j\in\overline{1, \lfloor\frac{m}2\rfloor} ;\\
  j-\lfloor\frac{m}2\rfloor, & j\in\overline{\lfloor\frac{m}2\rfloor+1,2\lfloor\frac{m}2\rfloor};\\
  j, & \mbox{ otherwise}.
                 \end{array}
\right.
\]

% Let $'$ be the involution of $\mathbf{I}$ such $i'=i+r$ for $i\in\overline{1, r}$ and $i'=i$ for $i\in\mathbf{I}_{1}.$ We also use the mapping $\sigma:\mathbf{I}\longrightarrow \{1, -1\}$ given by
%$\sigma(i)=-1$ for $i\in \overline{r+1, 2r}$ and $\sigma(i)=1$ otherwise.

Suppose $m=2r$ is even.
Define an  even linear mapping $D_{H}: \mathcal{O}(m,n)\longrightarrow
W(m,n)$  by
$D_{H}(a)=\sum_{i\in
\mathbf{I}}\sigma(i)(-1)^{|\partial_i||a|}\partial_i(a)\partial_{i'}.$
Put
$$\overline{H}(m,n)=\mathrm{span}_{\mathbb{F}}\{D_{H}(a)\mid
a\in\mathcal{O}(m,n)\}.$$ Write $\bar {\mathcal{O}}(m,n) $ for the quotient
superspace $\mathcal{O}(m,n)/\mathbb{F}\cdot 1$. We can
view $D_H$ as a linear operator of $\bar {\mathcal{O}}(m,n)$ since the kernel of $D_H$ is $\mathbb{F}\cdot1$. Thus we have
 $\overline{H}(m,n)\cong(\bar
{\mathcal{O}}(m,n), [\;,\;])$, where the
bracket is:
 $$ [a,b]= D_H\left(a\right)\left(b\right) \ \mbox{ for }
   a,b\in\bar{\mathcal{O}}(m,n).$$
Its derived algebra
${H}(m,n)$
 is  simple,  called  \textit{Hamiltonian Lie superalgebra}.

Suppose $m=2r+1$ is odd.
Define an  even linear mapping  $D_{K}: \mathcal{O}(m,n)\longrightarrow
W(m,n)$ by
%\begin{equation*}
%\begin{split}
$D_{K}(a)
%:&=\sum_{i\in \mathbf{I}\backslash\{m\}}(-1)^{|\partial_i||a|}\Big(x_i\partial_m(a)+\sigma(i')\partial_{i'}(a)\Big)\partial_i
%+\Big(2a-\sum_{i\in \mathbf{I}\backslash\{m\}}x_i\partial_i(a)\Big)\partial_m\\
=D_H(a)+\partial_{m}(a)\mathfrak{D}+\Delta(a)\partial_{m},$
%\end{split}
%\end{equation*}
where $\mathfrak{D}=\sum
_{i\in \mathbf{I}\backslash\{m\}}x_{i}\partial_{i}$ and $\Delta(a)=2a-\mathfrak{D}(a)$.
Put
$$\overline{K}(m,n)=\mathrm{span}_{\mathbb{F}}\{D_K(a)\mid a\in\mathcal{O}(m,n)\}.$$
Since $D_K$ is injective, we have
$\overline{K}(m,n)\cong(\mathcal{O}(m,n),[\;,\;]),$
where the bracket is:
\begin{equation*}
\begin{split}
[a,b]
&=D_H(a)b+\Delta(a)\partial_{m}(b)-\partial_{m}(a)\Delta(b)\ \mbox{ for } a, b\in{\mathcal{O}}(m,n).
\end{split}
\end{equation*}
Its derived algebra
$
K(m,n)$
 is  simple,  called  \textit{contact Lie superalgebra}.

 For simplicity, hereafter,  we write
$X$ for  $X(m,n)$, where $X=\mathcal{O}$,  $\bar{\mathcal{O}}$, $W$, $\overline{S}$, $S$, $\overline{H}$, $H$, $\overline{K}$ or $K$.  Let us consider the standard $\mathbb{Z}$-grading  of $L$, where $L=\mathcal{O}$, $W$, $S$, $\overline{H}$,  $H$  or $K$.
 Define the $\mathbb{Z}$-degrees of $x_i$ and $\partial_i$ to be
$\mathrm{zd}(x_i)=-\mathrm{zd}(\partial_i)=1+\delta_{L=K}\delta_{i=m}$, $i\in\mathbf{I}$.
 \textit{Hereafter, the symbol   $\mathrm{zd}(x)$
always implies that $x$ is
a $\mathbb{Z}$-homogeneous element.} Put $\xi= (m+\delta_{L=K})(p-1)+n$.
Then we have:
\begin{align*}
&  \mathcal{O}=\oplus_{i=-1}^{\xi}\mathcal{O}_i,\  \   \mathcal{O}_i=\mathrm{span}_{\mathbb{F}}\{f\in\mathcal{O}\mid\mathrm{zd}(f)=i\};  \\
&  W=\oplus_{i=-1}^{\xi-1}W_i,\  \   W_i=\mathrm{span}_{\mathbb{F}}\{f\partial_j\mid f\in\mathcal{O}_{i+1}, \ j\in\mathbf{I}\};\\
&  S=\oplus_{i=-1}^{\xi-2}S_i,\   \  S_i=\mathrm{span}_{\mathbb{F}}\{D_{jk}(f)\in W \mid f\in\mathcal{O}_{i+2}, \ j, k\in\mathbf{I}\};\\
&  \overline{H}=\oplus_{i=-1}^{\xi-2}H_i,\  \  H_i=\mathrm{span}_{\mathbb{F}}\{f\mid f\in\mathcal{O}_{i+2}\};\
H=\oplus_{i=-1}^{\xi-3}H_i;\\
&{K}= \oplus_{i=-2}^{\xi-2-\delta_{n-m-3\equiv0 \;( \mathrm{mod}\;p)}}K_i,\ \   K_i=\mathrm{span}_{\mathbb{F}}\{f\mid f\in\mathcal{O}_{i+2}\}.
\end{align*}
 We adopt the following conventions:
\begin{itemize}
\item[(1)] $L=H$ implies that $m=2r$ is even; $L=K$ implies that $m=2r+1$ is odd.
\item[(2)] $K$ can be viewed as a $\mathbb{Z}$-graded subalgebra
of $W$ when $\mathrm{zd}(x_m)=-\mathrm{zd}(\partial_m)=2$ for $W$.  Thus,  $L$ is a $\mathbb{Z}$-graded subalgebra
of $W$, where $L=S$, $H$ or $K$.
\item[(3)] For $L=K$, we write $z$ for $x_{m}.$%\item[(3)] We may assume that $H$ is naturally embedded in $K$.
\item[(4)] Write $\mathrm{alg}(S)$ for the subalgebra of $L$  generated by a subset $S$.
\end{itemize}
 A proper subalgebra $M$ of a $\mathbb{Z}$-graded Lie superalgebra  $L$ is called a
\textit{maximal graded subalgebra} (MGS) provided that  $M$ is
 $\mathbb{Z}$-graded and no nontrivial $\mathbb{Z}$-graded subalgebras  of $L$
strictly contains $M$.
%Let us give a rough classification of the MGS of $L$.
Since $L_{-1}$ is an irreducible $L_{0}$-module, it is clear that
$\oplus_{i\geq0}L_{i}$ is an MGS of $L$.  Any other MGS, $M$,  must satisfy exactly one of the
following conditions:
\begin{itemize}
\item[$\mathrm{(I)}$] $M_{-1}=L_{-1}$  and  $ M_{0}=L_{0};$
\item[$\mathrm{(II)}$]  $M_{-1} $ is a nontrivial proper subspace of $L_{-1};$
\item[$\mathrm{(III)}$] $ M_{-1}=L_{-1}$ and $M_{0}\neq L_{0}.$
\end{itemize}
Let $\frak{G}_0$ be a  subalgebra of $L_0$.   $\frak{G}_{0}$ is called \textit{reducible}
(\textit{resp.  irreducible}) if the $\frak{G}_{0}$-module $L_{-1}$ is reducible (resp.
irreducible). An MGS $\frak{G}=\sum_{i\geq-2}\frak{G}_i$ of $L$ is called
\textit{maximal reducible graded} (\textit{resp. maximal irreducible graded}) if
$\frak{G}_{0}$  is reducible (resp.
irreducible).

%%%%%%%%%%%%%%%%%%%%%%%%%%%%%%%%%%%%%%%%%%%%%%%%%%%%%%%%%%%%%%%%%%%%%%%%%%%%%%%%%%%%%%%%%%%%%%%%%%%%%%%%%%%%%%%%%%%%%%%%%%%%%%%%%%%%%%%%%%%%%%%%%%%%%%%%%%%%%
%%%%%%%%%%%%%%%%%%%%%%%%%%%%%%%%%%%%%%%%%%%%%%%%%%%%%%%%%%%%%%%%%%%%%%%%%%%%%%%%%%%%%%%%%%%%%%%%%%%%%%%%%%%%%%%%%%%%%%%%%%%%%%%%%%%%%%%%%%%%%%%%%%%%%%%%%%%%%
%%%%%%%%%%%%%%%%%%%%%%%%%%%%%%%%%%%%%%%%%%%%%%%%%%%%%%%%%%%%%%%%%%%%%%%%%%%%%%%%%%%%%%%%%%%%%%%%%%%%%%%%%%%%%%%%%%%%%%%%%%%%%%%%%%%%%%%%%%%%%%%%%%%%%%%%%%%%%
\section{Preliminary Results}
In order to simplify our considerations, in this section, we
establish some technical lemmas.  For $L=H$ or $K$,
we redescribe  $L$ in an appropriate form and
establish  a suitable automorphism of $L$ by virtue of a nondegenerate skew supersymmetric bilinear form on $L_{-1}$.

 As in the case of Lie superalgebras of characteristic 0 \cite{Kac1} or modular Lie algebras \cite{KS,St,SF}, it is easy to show the following:
\begin{lemma}\label{dibu}  Let $L=W, S, H$ or $K$.
\begin{itemize}
\item[$(1)$] $L$ is transitive.
\item[$(2)$] $L$ is generated by its local part,
 $L=\mathrm{alg}(L_{-1}+L_{0}+L_{1})$.
 \item[$(3)$] For the null of $L$, the following conclusions hold:
 \begin{align*}
   & W(m, n)_0\cong \frak{gl}(m, n);\
 S(m, n)_0\cong \frak{sl}(m, n); \\
   & H(2r, n)_0\cong \frak{osp}(2r, n);\
 K(2r+1, n)_{0}\cong \frak{osp}(2r, n)\oplus\mathbb{F}I_{2r+n}.
 \end{align*}
\end{itemize}
\end{lemma}
When $L=W$ or $S$,  we know that $L_{-1}$ is spanned  by the standard ordered $\mathbb{F}$-basis
\begin{align}
\{\partial_i\mid i\in\mathbf{I}\}.\label{basis 2}\end{align}
For a $\mathbb{Z}_2$-graded subspace $V=V_{\bar{0}}\oplus
V_{\bar{1}}$ of $L_{-1}$,   the super-dimension is denoted by $$\mathrm{superdim}V=(\mathrm{dim}V_{\bar{0}},\mathrm{dim}V_{\bar{1}}).$$

When $L=H$ or $K$, we redescribe $L$ in a desired form.  For  $i\in\mathbb{N}\backslash{\{0\}}$, write $A_i$ for an  $i\times i$ matrix, and particularly, let $I_i$ be the $i\times i$ unit matrix.
Denote by $\sqrt{a}$  a fixed solution of the equation $x^2=a$ in $\mathbb{F}$, where $a=-1, 2$.
Put
\begin{equation*}
y_i= \left\{\begin{array}{lll}
{x_i}, &i\in\mathbf{I}_0\cup \overline{m+2q+1, m+n};\\
\frac{x_i+\sqrt{-1} x_{i+q}}{\sqrt{2}}, &i\in \overline{m+1, m+q};\\
\frac{x_{i-q}-\sqrt{-1} x_{i}}{\sqrt{2}}, &i\in \overline{m+q+1, m+2q},
\end{array}\right.\end{equation*}
 where $0\leq d\leq n$, $q=\lfloor\frac{n-d}2\rfloor$. Then there exists an invertible matrix $A_{m+n}$  such that $(y_1, \ldots, y_{m+n})A=(x_1, \ldots, x_{m+n})$.
Obviously, $|y_i|=|x_i|$ and $\mathrm{zd}(y_i)=\mathrm{zd}(x_i),$ $i\in\mathbf{I}$.
  By  \cite[Lemma 2.5]{LZA}, we have:
$$\{y^{(\alpha)}y^u\mid \alpha\in \mathbf{A}(m), \
u\in\mathbf{B}(n)\}$$
is an $\mathbb{F}$-basis of $\mathcal{O}$, where $y^{(\alpha)}=x^{(\alpha)}$
and $y^u=y_{i_1}y_{i_2}\cdots y_{i_k}$ when $u=\langle i_1,i_2,\ldots,i_k\rangle$.
The basis-element $y^{(\alpha)}y^u$ is called \textit{a monomial}.

Write $(D_1, \ldots, D_{m+n})=(\partial_1, \ldots, \partial_{m+n})A^{t}$. Then we have
\[D_i= \left\{\begin{array}{lll}
\partial_i, &i\in\mathbf{I}_0\cup \overline{m+2q+1, m+n};\\
\frac{\partial_i-\sqrt{-1} \partial_{i+q}}{\sqrt{2}}, &i\in \overline{m+1, m+q};\\
\frac{\partial_{i-q}+\sqrt{-1} \partial_{i}}{\sqrt{2}}, &i\in \overline{m+q+1, m+2q}.
\end{array}\right.\]
By a direct computation, we have:
\[D_i(y_j)=\delta_{i=j}, \ \
{\sum}_{i\in \mathbf{I}\backslash\{2r+1\}}y_{i}D_{i}=\mathfrak{D}.\]
When $m=2r$, define   an  even linear mapping $E_{H}: \mathcal{O}\longrightarrow
W$ by
$$E_{H}(a)={\sum}_{i\in
\mathbf{I}}\sigma(i)(-1)^{|D_i||a|}D_i(a)D_{\widetilde{i}}, $$
where
\[\widetilde{i}= \left\{\begin{array}{lll}
i', &i\in\mathbf{I}_0\cup \overline{m+2q+1, m+n};\\
i+q, &i\in \overline{m+1, m+q};\\
i-q, &i\in \overline{m+q+1, m+2q}.
\end{array}\right.\]
%Obviously, $D_H$ and $E_H$ are $\mathcal{O}$-module derivations.
%which implies that
When $m=2r+1$, define    an  even linear mapping $E_{K}: \mathcal{O}\longrightarrow
W$  by
\begin{equation*}
\begin{split}
E_{K}(a)
%:&=\sum_{i\in \mathbf{I}\backslash\{m\}}(-1)^{|\partial_i||a|}\Big(x_i\partial_m(a)+\sigma(i')\partial_{i'}(a)\Big)\partial_i
%+\Big(2a-\sum_{i\in \mathbf{I}\backslash\{m\}}x_i\partial_i(a)\Big)\partial_m\\
&=E_H(a)+D_{m}(a)\mathfrak{D}+\Delta(a)D_{m}.
\end{split}
\end{equation*}
A direct computation shows that $D_L=E_L$.
Note that  $L_{-1}$ is spanned  by the standard ordered $\mathbb{F}$-basis
\begin{align}
\{y_i\mid  i\in \overline{1, 2r}\cup\overline{m+1, m+n}\}. \label{basis 1}
\end{align}
Define an even bilinear form $\beta: L_{-1}\times L_{-1}\longrightarrow \mathbb{F}$ satisfying
\begin{equation*}\label{mgseq121215}
\beta(u, v)={\sum}_{i\in
\mathbf{I}}\sigma(i)(-1)^{|D_i||u|}D_i(u)D_{\widetilde{i}}(v)\ \mbox{ for } u, v\in L_{-1}.
 \end{equation*}
 Then the matrix of $\beta$ in the ordered basis (\ref{basis 1}) is
\begin{equation}\label{basismgs120910 1}
J=\left(
\begin{tabular}{c|c}
  % after \\: \hline or \cline{col1-col2} \cline{col3-col4} ...
  \begin{tabular}{c|c}
    % after \\: \hline or \cline{col1-col2} \cline{col3-col4} ...
    0 & $I_r$ \\
    \hline
    $-I_r$ & 0 \\
     \end{tabular}
   & 0 \\
   \hline
  0 & \begin{tabular}{c|c|c}

        % after \\: \hline or \cline{col1-col2} \cline{col3-col4} ...
        0 & $-I_q$ & 0 \\\hline
        $-I_q$ & 0 & 0 \\\hline
       0 & 0 & $-I_{n-2q}$
      \end{tabular}
\end{tabular}
\right).
\end{equation}
Clearly, $\beta$ is a nondegenerate skew supersymmetric  bilinear form on  $L_{-1}$.

An $\mathbb{F}$-basis of $L_{-1}$ in which  the matrix of $\beta$ is $J$ is called \textit{generalized orthosymplectic}.
Let $V=V_{\bar{0}}\oplus V_{\bar{1}}$ be a subspace of $L_{-1}$. Suppose   $2a$ (resp. $d$) is the rank of $\beta$ restricted to  $V_{\bar{0}}$ (resp. $V_{\bar{1}}$).
 A  $\mathbb{Z}_2$-homogeneous basis of $V$
\begin{eqnarray}\label{basis 2}
\{e_1, \ldots, e_a, e_{r+1}, \ldots, e_{r+a}; e_{a+1},\ldots, e_{b}\mid e_{m+1}, \ldots, e_{m+c}; e_{m+n-d+1},\ldots, e_{m+n}\}
\end{eqnarray}
is called a \textit{$\beta$-basis} of $V$, if
 $$\{e_1, \ldots, e_a, e_{r+1}, \ldots, e_{r+a}; e_{a+1},\ldots, e_{b}\}, \ 0\leq a\leq b\leq r$$ is an $\mathbb{F}$-basis
 of $V_{\bar{0}}$ satisfying
 \[
\beta(e_i, e_j)=-\beta(e_j, e_i)=\left\{\begin{array}{ll}
  1,  & 1\leq i\leq a, j=\widetilde{i}; \\
  0,  & \mbox{otherwise}
\end{array}\right.
\]
 and  $$\{e_{m+1}, \ldots, e_{m+c}; e_{m+n-d+1},\ldots, e_{m+n}\}, \ 0\leq d\leq n,\   0\leq c\leq \lfloor\frac{n-d}{2}\rfloor$$
 is
an $\mathbb{F}$-basis  of $V_{\bar{1}}$
 satisfying
  \[
\beta(e_i, e_j)=\beta(e_j, e_i)=\left\{\begin{array}{ll}
  -1,  & m+n-d+1\leq i=j\leq m+n; \\
  0,  & \mbox{otherwise}.
\end{array}\right.
\]
The 4-tuple $(a, b, c, d)$ is called \textit{the $\beta$-dimension} of $V$, denoted by $\beta$-$\dim V=(a, b, c, d)$.
$V$ is  \textit{nondegenerate} (with respect to $\beta$)   if $a=b$ and $c=0$. $V$ is \textit{isotropic} if  $a=0$ and $d=0$.
Clearly, for any $\mathbb{Z}_2$-graded subspace of $L_{-1}$, there exists a $\beta$-basis of it, which  can extend to a generalized orthosymplectic basis of $L_{-1}$.

Now, suppose  $L=W$, $S$, $H$ or $K$. Put
$$\frak{V}^L=\{V\mid V \mbox{ is a nontrivial subspace of } L_{-1}\}.$$
$V\in\frak{V}^L$ is called \textit{a standard element } if $V$ is spanned by
$$\{\partial_{1},\ldots,\partial_k\mid \partial_{m+1},\ldots,\partial_{m+l}\},$$
when $L=W$ or $S$, $0\leq k\leq m$, $0\leq l\leq n$; if $V$ is spanned by
$$\{y_1, \ldots, y_a, y_{r+1}, \ldots, y_{r+a}; y_{a+1},\ldots, y_{b}\mid y_{m+1}, \ldots, y_{m+c}; y_{m+n-d+1},\ldots, y_{m+n}\},$$
when $L=H$ or $K$,  $0\leq a\leq b\leq r$, $ 0\leq d\leq n$,  $ 0\leq c\leq \lfloor\frac{n-d}{2}\rfloor$.
\textit{Hereafter, for $V, V'\in\frak{V}^L$, the symbol  $V\cong V'$  always means $\mathrm{superdim}V=\mathrm{superdim}V'$ when $L=W$ or $S$ and means $\beta$-$\dim V=\beta$-$\dim V'$ when $L=H$ or $K$}.

\begin{lemma}\label{mgsl3}
Let $L=W, S, H$ or $K$. Suppose $V $, $V'\in\frak{V}^L$ satisfying $V\cong V'$. Then there exists
a $\mathbb{Z}$-homogeneous automorphism ${\Phi}_L$ of $L$ such that ${\Phi}_L(V)=V'$.
\end{lemma}
\begin{proof}
Without loss of generality, we may assume that $V$ is a standard element in $\frak{V}^L$.
When $L=W$ or $S$,  suppose $\mathrm{superdim}V=\mathrm{superdim}V'=(k,l).$ Let $(E_{1},\ldots,E_{k} \mid E_{m+1},\ldots,E_{m+l})$ be a
$\mathbb{Z}_{2}$-homogeneous basis of $V'$. It extends to a
$\mathbb{Z}_{2}$-homogeneous basis of $W_{-1}$:
$$(E_{1},\ldots,E_{m} \mid E_{m+1},\ldots,E_{m+n}),$$ where
$|E_{i}|=|\partial_{i}|$, $i\in\mathbf{I}$. There exists an
 even invertible matrix $A_{m+n}$ such that
\begin{eqnarray}\label{a}
(E_{1},\ldots,E_{m+n})=(\partial_{1},\ldots,\partial_{m+n})A^{t}.
\end{eqnarray}
Let $(\xi_{1},\ldots,\xi_{m+n})=(x_{1},\ldots,x_{m+n})A^{-1}$. Consider the mapping $\phi$
such that $$\phi(x_{i})=\xi_{i}\quad \mbox{for all } i\in \mathbf{I}.$$
Notice that
$|x_{i}|=|\xi_{i}|$, since $A$ is even.
 By  \cite[Lemma 2.5]{LZA}, $\phi$ can  extend to an
endomorphism of $\mathcal{O}$, which is still written as $\phi$.
Then we have:
%\begin{eqnarray*}\label{b}
%(\phi(x_{1}),\ldots,\phi(x_{m+n}))=(x_{1},\ldots,x_{m+n})A^{-1}.
%\end{eqnarray*}
%It follows that
\begin{eqnarray}\label{c}
(\phi^{-1}(x_{1}),\ldots,\phi^{-1}(x_{m+n}))=(x_{1},\ldots,x_{m+n})A
\end{eqnarray}
We  denote by ${\Phi}$ the automorphism of $W$ which is induced by $\phi$ according to the formula
$${\Phi}(D)=\phi D\phi^{-1}\ \mbox{ for } \ D\in W.$$
Clearly, ${\Phi}$ is $\mathbb{Z}$-homogeneous.
By (\ref{a}) and (\ref{c}), we have:
\begin{eqnarray}\label{e}
\Phi(\partial_{i})=\phi\partial_{i}\phi^{-1}=E_{i}\quad\mbox{for all } i\in\mathbf{I}.
\end{eqnarray}
%Indeed, it follows from
%$$\phi\partial_{i}\phi^{-1}(x_{j})=E_{i}(x_{j}),\quad
%j\in\mathbf{I},$$ which can be verified by (\ref{a}) and (\ref{c}).
Furthemore, for $D=\sum_{i\in\mathbf{I}}f_{i}\partial_{i}\in W$, one can
verify that
$${\Phi}(D)=\phi
D\phi^{-1}={\sum}_{i,j\in\mathbf{I}}\partial_{i}(\phi^{-1}(x_{j}))\phi(f_{i})\partial_{j}.$$
%where $|\partial_j|=|\partial_i|$, since $\Phi$ is even.
By virtue of (\ref{a}) and (\ref{e}), we have:
$$\mathrm{div}({\Phi}(D))=\phi(\mathrm{div}D).$$
This shows that ${\Phi}({S})={S}$ since
$S=[\overline{S},\overline{S}]$.
Then $\Phi_L=\Phi\big|_L$ is desired.

When $L=H$ or $K$, suppose $\beta$-$\dim V=\beta$-$\dim V'=(a, b, c, d)$. Let
$\{e_i\mid  i\in \overline{1, 2r}\cup\overline{m+1, m+n}\}$ be an extension of $\beta$-basis (\ref{basis 2}) of $V'$
to a  generalized orthosymplectic basis of $L_{-1}$.
 Then, there exist two even invertible matrices
\begin{eqnarray*}
A=\left(\begin{array}{c|c}
  A_{2r} & 0 \\\hline
  0 & A_{n}
 \end{array}
 \right)\mbox { and }
A'=\left(\begin{array}{c|c}
 \begin{array}{c|c}
 A_{2r} & 0 \\\hline
 0 & I_1
  \end{array}
  & 0 \\\hline
 0 & A_n
  \end{array}
\right)
\end{eqnarray*}
satisfying
\begin{eqnarray*}
&&(e_1,\ldots,e_{2r}\mid e_{m+1}, \ldots,e_{m+n})A=(y_1,\ldots,y_{2r}\mid y_{m+1},\ldots,y_{m+n}),\\
&&(e_1,\ldots,e_{2r},e_{2r+1}\mid e_{m+1},\ldots,e_{m+n})A'=(y_1,\ldots,y_{2r},y_{2r+1}\mid y_{m+1},\ldots,y_{m+n}).
\end{eqnarray*}
Thus, we obtain that
\begin{eqnarray}\label{mgseq1}
A^{-1}J(A^{t})^{-1}=J.
\end{eqnarray}
By virtue of  \cite[Lemma 2.5]{LZA}, there exists a unique automorphism of $\mathcal{O}$ denoted by $\phi_L$ satisfying
$\phi_L(y_i)=e_i$, $i\in\mathbf{I}.$
As in the case $L=W$, we  denote by $\overline{{\Phi}}_L$ the $\mathbb{Z}$-homogeneous automorphism of $W$ which is induced by $\phi_L$.
 From  (\ref{mgseq1}), we have:
\begin{eqnarray}
&(\overline{{\Phi}}_H(D_1),\ldots,\overline{{\Phi}}_H(D_{m+n}))=(D_1, \ldots, D_{m+n})A^{t}\label{mgseq2}.\\
&(\overline{{\Phi}}_K(D_1),\ldots,\overline{{\Phi}}_K(D_{m+n}))=(D_1, \ldots, D_{m+n})A'^{t}\label{mgseq3}.
\end{eqnarray}
For any $D=\sum_{i\in \mathbf{I}\backslash\{2r+1\}}f_iD_i\in W$ and $fD_{2r+1}\in W$, from (\ref{mgseq1})-(\ref{mgseq3})  we have:
\begin{align}
\label{mgseq4}
&\overline{{\Phi}}_L(D)={\sum}_{i\in \mathbf{I}\backslash\{2r+1\}}\phi_L(f_i)\overline{{\Phi}}_L(D_i),\\
&\overline{{\Phi}}_K(fD_{2r+1})=\phi_K(f)D_{2r+1}.
\end{align}
For any $f\in \mathcal{O}$, we have:
\begin{eqnarray}
&&\overline{{\Phi}}_K(D_{2r+1}(f)\mathfrak{D})=D_{2r+1}(\phi_K(f))\mathfrak{D}.\label{mgseq5}\\
&&\overline{{\Phi}}_K((2-\mathfrak{D})(f)D_{2r+1})=(2-\mathfrak{D})\phi_K(f)D_{2r+1}.\label{mgseq6}
\end{eqnarray}
By virtue of   (\ref{mgseq1})-(\ref{mgseq6}), we have:
$$\overline{{\Phi}}_L(E_L(f))=E_L(\phi_L(f)) \ \mbox{ for any } f\in\mathcal{O}.$$
It follows that $\overline{\Phi}_L({L})={L}$ since
 $L=[\overline{L}, \overline{L}]$.
Then $\Phi_L=\overline{\Phi}_L\big|_L$ is desired.
\end{proof}
For convenience, we introduce the following notations.
Let $L=W$ or $S$. For any $V\in\frak{V}^L$  with $\mathrm{superdim}V=(k, l)$,
put
\begin{eqnarray*}
 \mathbf{I}(k,l)=
\overline{1,k}\cup \overline{m+1,m+l},\qquad
 \overline{\mathbf{I}}(k,l)=\mathbf{I}\backslash
\mathbf{I}(k,l),
\end{eqnarray*}
If $V$ is  standard, we have:
\begin{equation}V=\mathrm{span}_{\mathbb{F}}\{\partial_i\mid i\in \mathbf{I}(k,l)\}.\label{eqsb}\end{equation}
Let $L=H$ or $K$. For any $V\in\frak{V}^L$  with $\beta$-$\dim V=(a, b, c, d)$, put
\begin{align}
%\label{mgseq12733}
&I_{01}=\overline{1, a}; \ \bar{I}_{01}=\overline{r+1, r+a}; \ I_{02}=\overline{a+1, b}; \ \bar{I}_{02}=\overline{r+a+1, r+b};\nonumber\\
&{I}_{11}=\overline{m+n-d+1, m+n};\ {I}_{12}=\overline{m+1, m+c};\ \bar{I}_{12}=\overline{m+q+1, m+q+c};\nonumber\\
&{I}_{03}=\overline{b+1, r}\cup\overline{r+b+1, m}; \ {I}_{13}=\overline{m+c+1, m+q}\cup\overline{m+q+c+1, m+n-d}.\label{12731eq1}
\end{align}
Obviously,  $\mathbf{I}=J_1\cup J_2\cup \bar{J}_2\cup J_3,$ where
\begin{align}\label{mgs120924eq1}
J_1={I}_{01}\cup\bar{I}_{01}\cup{I}_{11}, \ \ J_2={I}_{02}\cup {I}_{12}, \ \ \bar{J}_2=\bar{I}_{02}\cup \bar{I}_{12}\ \mbox{ and }\
J_3={I}_{03}\cup {I}_{13}.
\end{align}
We call $J_i$ to be \textit{single} (resp. \textit{twinned})
if $\mathbf{I}_0\cap J_i=\emptyset$ and there exists only one element in  $\mathbf{I}_1\cap J_i$
(resp. there exist two elements in  $\mathbf{I}_1\cap J_i)$, $i=1, 3$.
 If $V$ is  standard, we have:
\begin{equation}V=\mathrm{span}_{\mathbb{F}}\{y_i\mid i\in J_1\cup J_2\}.\label{eqsb2}\end{equation} For any $i\in \mathbf{I}$, let us assign to each $y_i$
a value as follows:
\begin{equation}\label{mgseq12734}
\nu(y_i)=\left\{\begin{array}{llll}
  1  & i\in J_1; \\
  0 & i\in J_2; \\
  \frac{1}3  & i\in \bar{J}_2; \\
  2 & i\in J_3.
\end{array}\right.
\end{equation}
If $a=y_1^{\alpha_1}y_2^{\alpha_2}\cdots y_m^{\alpha_m}y^u$, define $\nu(a)=\prod_{i\in \mathbf{I}_0}\nu(y_i)^{\alpha_i}\prod_{i\in u}\nu(y_i).$

%\begin{lemma}\label{mgs1276l3} Suppose  $\Phi$ is a $\mathbb{Z}$-homogeneous automorphism of $L$. Then
%$\Phi(\mathcal{M}_i(V))=\mathcal{M}_i(\Phi(V))$ for all $i\geq -2$. Moreover, $\Phi(\mathcal{M}(V))=\mathcal{M}(\Phi(V)).$
%\end{lemma}

\begin{remark}\label{mgsr2} Let $T$ be a torus of $L$, $L=W, S, H$ or $K$.
 Consider the weight space decompositions with respect to $T$:
\begin{eqnarray*}
L=L^{\theta}\oplus\oplus_{\gamma\in \Delta}L^{\gamma}, \ \ L_i=L_i^{\theta}\oplus\oplus_{\gamma\in \Delta_i}L^{\gamma}_i,
\end{eqnarray*}
where $\Delta_i\subset\Delta\subset T^*$ and $\theta$  is the zero weight. Notice the standard facts below.
\begin{itemize}
\item[$(1)$] For $t\in T$,  suppose  $x=x_1+x_2+\cdots+x_n\in L$ is a sum of eigenvectors of
$\mathrm{ad}t$ associated with mutually distinct eigenvalues. Then all $x_i$'s lie in $\mathrm{alg}(\{t, x\}) $.
\item[$(2)$]
$T=\mathrm{span}_{\mathbb{F}}\{y_iy_{\widetilde{i}}\mid i\in \overline{1, r}\cup\overline{m+1, m+q}\}$
is a torus  of $L$, $L=H$ or $K$,
 where $0\leq d\leq n$, $q=\lfloor\frac{n-d}2\rfloor$.
 Define  $\epsilon_j$ to  be  the  linear function on $T$ by
$$\epsilon_j(y_iy_{\widetilde{i}})=\delta_{j {\widetilde{i}}}-\delta_{ji}.$$
For $i, j, k\in\mathbf{I}\backslash\{\{2r+1\}\cup\overline{m+2q+1, m+n}\}$, if $\epsilon_i+\epsilon_j\in\Delta_0$, we have:
\[
\dim L^{\epsilon_i}_{-1}=1;\ \ \dim L^{\epsilon_i+\epsilon_j}_{0}=1;
\ \ L^{\epsilon_k}_{0}={\sum}_{l= m+2q+1}^{ m+n}\mathbb{F}y_ky_l.
\]
\end{itemize}
\end{remark}

\section{MGS of Type (I)}\label{001}
To formulate the MGS of type $(\textrm{I})$, we introduce the following notations.

 For $i\geq 1$, write
\begin{equation}\label{L2}
L_{i}'=\overline{S}_i=\{D\in L_{i}\mid \mathrm{div} D=0\}\quad\mbox{and} \quad
L_{i}''=\{f\mathfrak{D}\mid f\in \mathcal {O}_{i}\},
\end{equation}
where  $L=W$ or $S$, $\mathfrak{D}$ is  the \textit{degree derivation} of $\mathcal{O}$; that is,
$\mathfrak{D}=\sum_{k\in \mathbf{I}}x_{k}\partial_{k}$.  Clearly, both $W_{i}'$ and
$W_{i}''$ are nontrivial subspaces of $W_{i}$.

For $i\geq 0$, write
$$K_{ij}=\big\{u\in K_i\mid u=fz^{j}, f\in\mathcal{O}_{i+2-2j}, [1, f]=0\big\}.$$
Clearly,   $K_{ij}$ is a  nontrivial subspace of $K_i$.

\begin{theorem}\label{thm01} All  MGS of type $\mathrm{(I)}$ are characterized as follows:
\begin{itemize}
\item[$\mathrm{(1)}$] If $m-n+1
\equiv0\pmod{p} $ then $W$ has exactly one MGS of type $\mathrm{(I)}:$
$$W_{-1}+W_{0}+W_{1}'+W_{2}'+\cdots+W_{\xi-2}'$$
with dimension $(m+n-1)2^np^{m}+2$;

If $m-n+1
\not\equiv0\pmod{p} $ then $W$ has exactly two MGS of type $\mathrm{(I)}:$
$$ W_{-1}+W_{0}+W_{1}''\quad\mbox{and}\quad
 W_{-1}+W_{0}+W_{1}'+W_{2}'+\cdots+W_{\xi-2}'$$
with dimensions $(m+n)(m+n+2)$ and $(m+n-1)2^np^{m}+2$, respectively.
\item[$\mathrm{(2)}$] If $m-n+1
\equiv0\pmod{p} $ then $S$ has exactly one MGS of type $\mathrm{(I)}:$
$$S_{-1}+S_{0}+S_{1}''$$
with dimension $(m+n)^2+2(m+n)-1$;

If $m-n+1
\not\equiv0\pmod{p} $ then $S$ has exactly one MGS of type $\mathrm{(I)}:$
$$S_{-1}+S_{0}$$
with dimension $(m+n)^2+(m+n)-1$.
 \item[$\mathrm{(3)}$]  $H$ has exactly  one MGS of type $\mathrm{(I)}:$
    $$H_{-1}+H_0 $$
    with dimension $(m+n)^2+m$.
\item[$\mathrm{(4)}$] $K$ has exactly two   MGS of type $\mathrm{(I)}:$
$$K_{-2}+K_{-1}+K_0+{\sum}_{i=1}^{2r(p-1)+n}K_{i0}\ \mbox{ and }\  K_{-2}+K_{-1}+K_0+K_{11}+K_{22}$$
with dimensions  $2^np^{2r}+1$ and $(2r+n)^2+4r+n+3$, respectively.
\end{itemize}
\end{theorem}

 We note that  many preliminary results in this section are analogous to the ones of Lie algebras (see \cite{HM,KS,St}).
We will need the following formulas which are easy to verify by direct calculations.
\begin{lemma}\label{div}
For $f\in\mathcal{O}_{s} $ and $g\in\mathcal{O}_{t},$
\begin{eqnarray}
&&\mathrm{div}(f\mathfrak{D})=(m-n+s)f \quad\mbox{for}~f\in\mathcal{O}_{s}, \nonumber
\\
&&[f\mathfrak{D},g\mathfrak{D}]=(t-s)fg\mathfrak{D}.\label{jisuanlm2}
\end{eqnarray}
\end{lemma}
\begin{lemma}\label{lem1.3} The following statements hold.
\begin{itemize}
\item[$(1)$]  $W_{s}'$ and $W''_{s}$ are $W_{0}$-submodules of $W_{s}$.
Moreover, $W_{s}''$ is  irreducible.
\item[$(2)$] If $m-n+s\not\equiv 0\pmod{p}$ then
$W_{s}=W_{s}'\oplus W_{s}''$;
\item[$(3)$] If
$m-n+s\equiv 0\pmod{p}$ then $W_{s}''\subset W_{s}'$.
\end{itemize}
\end{lemma}

\begin{proof}
  Note that $\mathrm{div}$ is a derivation from $W$ to $\mathcal{O}$ as $W$-module. Thus, (1), (2) and (3)  hold by virtue of  Lemma \ref{div}.
\end{proof}

Below, the $1$-component $W_{1}$ will be a focus of our attention. For convenience, we introduce two concepts, by which our arguments are largely simplified:
An element $\mathscr{L}$ in $W_{1}$ is called a \textit{leader} if it is of the form
\begin{eqnarray*}\label{Bdexingshi}
\mathscr{L}
=x_{1}^{2}\partial_{1}+{\sum}_{i=2}^{m+n}f_{i}\partial_{i} \quad\mbox{where}\; f_{i}\in \mathcal{O}_{2};
\end{eqnarray*}

An element in $W_{1}$ is called $1$-\textit{defective} if it is of the form
$${ \sum}_{i=2}^{m+n}f_{i}\partial_{i}\ \mbox{ where }  f_{i}\in \mathcal{O}_{2}.$$

\begin{lemma}\label{lembuchong1}
Let $D\in W_1$.
\begin{itemize}
\item[$(1)$]
  $[x_1\partial_{j},D]=0$ for all $j\geq 2$ if and only if $D=\lambda x_1\mathfrak{D}+x_{1}^{2}\sum_{j\geq 2}k_{j}\partial_{j}$ for some $\lambda, k_j\in \mathbb{F}.$
\item[$(2)$]  $[x_1\partial_{j},D]\in W_1''$ for all $j\geq 2$ if and only if $D=f\mathfrak{D}+x_{1}^{2}\sum_{j\geq 2}k_{j}\partial_{j}$ for some $f\in \mathcal{O}_1$ and $ k_j\in \mathbb{F}.$
    \end{itemize}
 \end{lemma}
 \begin{proof}  (1)
Suppose $[x_1\partial_{j},D]=0$ for all $j\geq 2$ and write $D=\sum_{i}a_{i}\partial_{i}$. Then
\begin{equation}\label{eqbuliu1}
x_1\partial_j(a_i)-\delta_{i=j}(-1)^{|x_1\partial_j||a_1\partial_1|}a_1=0\quad\mbox{for all}\; j\geq 2.
\end{equation}
Then $a_1=kx_1^2$ for some $k\in \mathbb{F}.$ If $k=0$, it follows from (\ref{eqbuliu1}) that
$\partial_j(a_i)=0$ for all $j,i\geq 2$. That is, $a_j=k_jx_1^2$ for all $j\geq 2.$ Hence $D=x_1^2\sum_{j\geq 2}k_j\partial_{j}.$
If $k\neq 0$ then write $a_{1}=x_1^2.$ From (\ref{eqbuliu1}) one deduces
$$
\partial_j(a_i)=\delta_{i=j}x_{1}\quad \mbox{for all}\; i, j\geq2.
$$
It follows that  $a_j=k_jx_1^2+x_1x_j$ for $j\geq2$ and one direction holds. The other one is clear.

(2) Write $[x_1\partial_{j},D]=f_j\mathfrak{D}$, $f_j\in \mathcal{O}_{1},$ $j\geq 2.$ By acting on $x_i$ with $i\neq 1,j,$ we have $x_1[\partial_j,D](x_i)=f_jx_i$ and then $f_j=k_jx_1$ for some $k_j\in\mathbb{F}.$ Thus \begin{equation}\label{liumbu39}
[x_1\partial_{j},D]=k_jx_1\mathfrak{D}\quad \mbox{for all}\; j\geq2.
\end{equation}
Since $[x_1\partial_j, x_i\mathfrak{D}]=\delta_{i=j}x_1\mathfrak{D},$  from (\ref{liumbu39}) we have
 $\left[x_1\partial_j, D-\left(\sum_{i\geq 2}k_ix_i\right)\mathfrak{D}\right]=0$ for all $ j\geq 2.$
 Now the conclusion follows from (1).
 \end{proof}
\begin{lemma}\label{B}
Let $M$ be a nonzero $W_{0}$-submodule of $W_{1}.$
\begin{itemize}
\item[$\mathrm{(1)}$] $M$ contains  a leader.
\item[$\mathrm{(2)}$] If $M$  contains a leader which does not lie in $W_{1}''$ then $M$ contains a nonzero 1-defective element.
\item[$\mathrm{(3)}$] If $M$ contains a nonzero $1$-defective element then $M\supset W_{1}'$. In particular,
as $W_{0}$-module, $W_{1}'$ is generated by $x_{1}^{2}\partial_{j}$ for any fixed $j\geq 2$.
\item[$\mathrm{(4)}$] As $W_{0}$-module, $W_{1}$ is generated by $x_1^{2}\partial_{1}.$
\item[$\mathrm{(5)}$] Any nonzero $W_{0}$-submodule of $W_1$ different from $W_1''$ must contain $W_1'.$
\end{itemize}
\end{lemma}
\begin{proof} (1), (3) and (4) need only a straightforward verification.

(2)   Let $D=x_{1}^{2}\partial_1+\cdots$ be a leader in $M\setminus W_{1}''$. Then $[x_1\partial_{j}, D]\in M$ are 1-defective for all
$j\geq 2.$ If they are not all zero, we are done. Otherwise, by Lemma  \ref{lembuchong1}(1),
$$\mbox{$D= x_1\mathfrak{D}+x_{1}^{2}\sum_{j\geq 2}k_{j}\partial_{j}$\quad for some $ k_j\in \mathbb{F}.$}$$
Clearly, $\sum_{j\geq 2}k_{j}\partial_{j}\neq 0,$ say, $k_{2}\neq 0.$
Consequently, $x_{1}^{2}\partial_{2}=k_{2}^{-1}[x_2\partial_2, D]\in M$.

(5) Let $M$ be a nonzero $W_{0}$-submodule and  $M\neq W_{1}'$. Let us show that $M\supset W_1'.$
By (1), (2) and (3) we may assume that  all the leaders of $M$ lie in $W''$. Then
$W''\subset M,$ since $W''$ as $W_{0}$-module is irreducible by Lemma \ref{lem1.3}(1). For $D
\in M\setminus W''$, if there is some $i\geq 2$ such that $E=[x_1\partial_{i}, D]\not\in W'',$ then $[x_1\partial_j, E]$ is a leader or 1-defective for any $j\geq 2$. By (3), one may assume that there is $D\in M\setminus W''$ which is pulled into $W''$ by any $x_1\partial_j$ with $j\geq 2.$  Then by Lemma \ref{lembuchong1}(2), $M$ contains a nonzero 1-defective element and then $M\supset W'.$
\end{proof}
\begin{lemma}\label{lemxin1.9} The following statements hold.
\begin{itemize}
\item[$\mathrm{(1)}$] $W_{1}'$ is a maximal
$W_{0}$-submodule of $W_{1}$.
\item[$\mathrm{(2)}$]  If  $m-n+1\not\equiv 0\pmod{p}$, $W_{0}$-module $W_{1}'$ is irreducible. In particular, $W_1$ has a decomposition of irreducible $W_{0}$-submodules:
$$
W_1=W_{1}'\oplus W_{1}''.
$$
\item[$\mathrm{(3)}$] If $m-n+1\equiv 0\pmod{p}$, $W_1$ has exactly a composition series of $W_{0}$-submodules: $$0\subset W_{1}''\subset W_{1}'\subset W_1.$$
\end{itemize}
\end{lemma}

\begin{proof}
(1) Let $M$ be a submodule of $W_{1}$  containing strictly
$W_{1}'$.   Note that
$$\mathrm{div}: \mathrm{span}_{\mathbb{F}}\{x_1x_1\partial_{1}, x_2x_1\partial_{1},\ldots,x_{m+n}x_1\partial_{1}\}\longmapsto \mathcal{O}_1$$  is  surjective.  Pick any $D\in M\setminus W_{1}'$. Then there exists $E=fx_1\partial_1,$ $f\in\mathcal{O}_1,$ such that
$\mathrm{div}E=\mathrm{div}D$. That is, $E-D\in W_1'\subset M$ and then $0\neq E\in M.$ If $\partial_j(f)=0$ for all $j\geq 2$ then  $E=\partial_1(f)x_{1}^{2}\partial_1$ and hence $M=W_1$ by Lemma \ref{B}(4). Suppose $\partial_j(f)\neq 0$ for some $j\neq 1.$ Then
$x_jx_1\partial_1= {\partial_j(f)^{-1}}[x_j\partial_j, E]\in M$. It follows that
$$x_1^{2}\partial_1-(-1)^{|\partial_{j}|}x_jx_1\partial_j=[x_1\partial_j, x_jx_1\partial_1]\in M.$$ Note that
$\frac{1}{2}x_1^{2}\partial_1-(-1)^{|\partial_{j}|}x_jx_1\partial_j$ is in $ W_1'\subset W.$  It follows that $x_1^2\partial_1\in M$ and $M=W_1$ by Lemma \ref{B}(4), showing that $W_1'$ is maximal.

(2)  and (3)   are immediate
  consequences of  Lemmas \ref{lem1.3} and \ref{B}(5).
\end{proof}
\begin{corollary}\label{1.10} The following statements hold.
\begin{itemize}
\item[$\mathrm{(1)}$]
If $m-n+1\not\equiv0\pmod{p}$ then $S_{1}$ is an irreducible
$S_{0}$-module.
\item[$\mathrm{(2)}$] If $m-n+1\equiv0\pmod{p}$ then $S_{1}''$ is the
unique nontrivial $S_{0}$-submodule of $S_{1}$.
\end{itemize}
\end{corollary}
\begin{proof} Note that $W_0=S_0+\mathbb{F}\mathfrak{D}$ and that
$W'_{1}=S_{1}$. If $m-n+1\equiv0\pmod{p}$ then $S_{1}''=W_{1}''$. The lemma  follows directly from  Lemma \ref{lemxin1.9}.
\end{proof}

\begin{lemma}\label{lem-liu-liu-Melikyan63}
$\{D\in W_{2}\mid [W_{-1},D]\subset W_{1}''\}=0.$
\end{lemma}
\begin{proof}
Write $D=\sum_{i\in\mathbf{I}}a_{i}\partial_{i} \in W_{2}$ and suppose $D$ is pulled into $W_{1}''$ by $W_{-1}$. Then, each $a_{i}$ must be a multiple of $x_{i}^{2}$ and in particular, $a_{j}=0$ for all $j> m.$ Write $D=\sum_{i\leq m}f_ix_{i}^{2}\partial_i$, where $f_i\in \mathcal{O}_1.$ Since $[\partial_j, D] \in W_1''$, one deduces that $\partial_j(f_i)=0$ for  $j>m$ and   $i\leq m.$ Now it is clear that $D$ is not in $ W_1''$ unless it is zero.
\end{proof}

Let
\begin{eqnarray*}
&& M'=W_{-1}+W_{0}+W_{1}'+W_{2}'+\cdots+W_{\xi-2}',\\
&& M''=W_{-1}+W_{0}+W_{1}''.
\end{eqnarray*}
 Using (\ref{jisuanlm2}) and keeping in mind that $\mathrm{div}$ is a derivation from $W$ to $\mathcal{O}$, one may verify that  $M'$ and $M''$
are subalgebras of $W$.
\begin{lemma}\label{llmbuch1616lem}
Suppose $M$ is a proper subalgebra containing $W_{-1}\oplus W_{0}\oplus W_{1}'.$ Then
$M\subset M'$.
 \end{lemma}
\begin{proof}  Assume conversely that $M\not\subset M'$. Then there exists $D\in M\cap\sum_{i\geq 1}W_{i}$ satisfying $\mathrm{div} D\not=0$.
Using the formula $\mathrm{div}[\partial_j, D]=\partial_j(\mathrm{div}D)$ for all $j\in\mathbf{I}$, one sees that
 $M\supset W_1$ by Lemma \ref{lemxin1.9}(1). By Lemma \ref{dibu}(2), $M=W$, a contradiction.
\end{proof}
\noindent
\textbf{Proof of (1) and (2) in Theorem \ref{thm01}}
(1)
\textit{Claim A}: $ M'$ is maximal. This follows immediately from Lemma \ref{llmbuch1616lem}.

 \textit{Claim B}: $ M''$ is maximal if $m-n+1\not\equiv0\pmod{p}.$ Let $M$ be a subalgebra  strictly containing $M''$. By transitivity and Lemma \ref{lem-liu-liu-Melikyan63}, $M\cap W_1$ must strictly
contain $W_1''$.  Lemma \ref{lemxin1.9}(2) forces
$M\supset W_1$ and therefore, $M=W$ by Lemma \ref{dibu}(2).

\textit{Claim C}:  $M'$ and $M''$ exhaust all the maximal  subalgebras of type (I). Let $M$ be a maximal
subalgebra of type (I). By transitivity, $M$ must contain a nonzero element of $W_1$ and therefore, $M\cap W_1\neq 0$ is a nonzero $W_0$-submodule of $W_1$. By Lemma \ref{B}(5), we have $M\cap W_1= W_1''$ or
$M\cap W_1\supset W_1'$.\\

\noindent\textbf{Case 1.} Suppose $m-n+1\not\equiv0\pmod{p}.$  If $M\cap W_1= W_1''$ then Claim B forces $M=M''.$
Suppose $M\cap W_1\supset W_1'$. By Lemma \ref{llmbuch1616lem}, we have $M\subset M'$ and then  $M=M'$ by
the maximality of $M$.\\

\noindent\textbf{Case 2.} Suppose $m-n+1\equiv0\pmod{p}.$ We have $M\supset W''$. Since $W''\subsetneq W'$ in this situation, one sees $M\supsetneq W''$. By transitivity and Lemma \ref{lem-liu-liu-Melikyan63}, $M\cap W_1\supsetneq W_1''$ and hence $M\cap W_1\supset  W_1'$ by Lemma \ref{B}(5). It follows from Lemma \ref{llmbuch1616lem} that $M=M'.$ This completes the proof of (1).

(2)  First of all, $S_{-1}+S_0$ and $S_{-1}+S_0+S_1''$ ($m+n-1\equiv 0\pmod{p}$) are subalgebras of $S$. Let $M$ be a maximal subalgebra of $S$  containing $S_{-1}+S_0$. Note that $W_1''=S_1''$ when  $m-n+1\equiv 0\pmod{p}$.
By the transitivity of $S$, Lemmas \ref{dibu}(2), \ref{lem-liu-liu-Melikyan63} and Corollary  \ref{1.10}, we obtain that
  $M=S_{-1}+S_0+S_1''$ when $m+n-1\equiv 0\pmod{p}$;  $M=S_{-1}+S_0$  when $m+n-1\not\equiv 0\pmod{p}$. The process shows also that these two subalgebras are indeed maximal. This completes the proof of (2).
\qed

\begin{remark}   For $W$ and $S$, the arguments for MGS of the other types will be reduced to the case of type $\mathrm{(I)}$ MGS by the method of  minimal counterexample.
\end{remark}

%\begin{remark}  One may consider furthermore structure properties for MGS of type (I). The interested reader is referred to Examples \ref{weishu2} and \ref{weishusss}.
%\end{remark}
\begin{lemma}
\label{mgsl2}  The following statements hold.
\begin{itemize}
\item[$\mathrm{(1)}$] $H_1$ is an irreducible $H_0$-module.
\item[$\mathrm{(2)}$] For $i\geq 0$, $K_i$ is a direct sum of $K_0$-submodules $K_{ij}$.
Moreover,   $K_{10}$ and $K_{11}$ are irreducible $K_0$-modules.
\end{itemize}
\end{lemma}
\begin{proof}
 Using the results in  the case of  modular Lie algebras \cite{KS} and by a direct computation, it is easy to show that  (1) holds. Since $K_{10}\cong H_1$ and $K_{11}\cong H_{-1}$ as $H_0$-modules, by  irreducibilities  of $H_{-1}$ and $H_1$, (2) holds.
\end{proof}
Let
\begin{eqnarray*}
&&M'=K_{-2}+K_{-1}+K_0+{\sum}_{i=1}^{2r(p-1)+n}K_{i0},\\
&&M''= K_{-2}+K_{-1}+K_0+K_{11}+K_{22}.
\end{eqnarray*}
By a standard and direct computation, one may verify that  $M'$ and $M''$
are subalgebras of $K$.\\

\noindent
\textbf{Proof of (3) and (4) in Theorem \ref{thm01}}
(3) This statement follows immediately from  Lemmas \ref{dibu}(2) and  \ref{mgsl2}(1).

(4) \textit{ Claim  A}:  $M'$ is maximal.
For any $0\not=u\in K$,  $u\not\in M'$, put $\overline{M}=\mathrm{alg}(M'+\mathbb{F}{u})$.
Note that there exist $k\in\mathbb{N}$ and $v_1,\ldots,v_s\in K_{-1}$ such that
$$0\not=u_{1}=fz+\alpha z^2=[v_1, [\cdots[v_s, (\mathrm{ad}1)^ku]\cdots]]\in \overline{M}\backslash M',$$
where $f \in \mathcal{O}$ satisfying  $[1,f]=0$ and  $\alpha\in\mathbb{F}$.
Then
 there exists $i\in\mathbf{I}$ such that $[y_{\widetilde{i}}, u_1]-y_{\widetilde{i}}f\not=0$. It follows that $$0\not=(\sigma(\widetilde{i})(-1)^iD_i(f)+\alpha y_{\widetilde{i}})z\in\overline{M}.$$
 From  Lemma \ref{dibu}(1), there exists a nonzero element in $\overline{M}\cap K_{11}.$
By Lemmas \ref{dibu}(2) and \ref{mgsl2}(2),  we have $\overline{M}=K$. Thus $M'$ is maximal.

\textit{Claim B}: $M''$ is maximal.
  For any $0\not=u\in K$,  $u\not\in M''$, put $\overline{M}=\mathrm{alg}(M''+\mathbb{F}{u})$.
 It is sufficient to show that there exists a nonzero element in $K_{10}\cap\overline{M}$.
When $\mathrm{zd}(u)>2$,
by  transitivity, there exist $v_1,\ldots,v_s\in K_{-1}$ such that
$$0\not=u_3=u_{30}+u_{31}+u_{32}=[v_1, [\cdots[v_s, u]\cdots]]\in\overline{M},$$
where $u_{3i}\in K_{3i}$, $i=0,1, 2$. Note that $[1, u_{32}]\in K_{11}$.
If $u_{31}\not=0$, then
$$0\not=[1, u_{31}]=[1, u_{3}-u_{30}-u_{32}]\in\overline{M}\cap K_{10}.$$
If $u_{31}=0$,   there exists $j\in\mathbf{I}$ such that
$$0\not=u_2=\sigma(\widetilde{j})(-1)^jD_j(u_{30})+y_{\widetilde{j}}D_{2r+1}(u_{32})\in\overline{M}\backslash{M''},$$
Note that  $\mathrm{zd}(u_2)=2$.
Thus, it remains  to consider the case
 $\mathrm{zd}(u)=2$. Assume that $u=u_{20}+u_{21}$, where $u_{2i}\in K_{2i}$, $i=0,1$.
If $u_{21}=0$ the conclusion follows.
Notice that ${H_{0}}\cong K_{21}$ as $H_0$-module.
If $u_{21}\not=0$, by Remark \ref{mgsr2} and a direct computation, we obtain that there exists $i\in \mathbf{I}_0$ such that $u_2=u'_{20}+y^{(2\varepsilon_i)}z\in\overline{M}$,
where $u'_{20}\in K_{20}.$
Since  $[K_{-1}, u_2]\subset\overline{M}$, there exists $j\in\mathbf{I}$ such that $$0\not=\sigma(\widetilde{j})(-1)^jD_j(u'_{20})+y_{\widetilde{j}}y^{(2\varepsilon_i)}\in\overline{M}\cap K_{10}.$$
Thus the conclusion holds.

\textit{Claim C}: $M'$ and $M''$  exhaust all  the maximal graded subalgebras of type $\mathrm{(I)}$.
Suppose $N$ is a maximal graded subalgebra of $K$ containing $K_{-1}+K_0$. By  transitivity, there exists
 $0\not=D=D_{10}+D_{11}\in N$, where $D_{1i}\in K_{1i}$, $i=0,1$.
 Since $N\subsetneq K$, we claim that $D_{11}=0$ or $D_{10}=0$.
 Indeed, if $D_{11}\not=0$ and $D_{10}\not=0$,  by the irreducibility of $K_{10}$, we have
 $$w=y^{(2\varepsilon_1)}y_{\widetilde{1}}+{\sum}_{i=1}^{2r+n}\alpha_iy_iz\in N,\ \alpha_i\in\mathbb{F}.$$
 We consider the following cases.\\

 \noindent\textbf{Case 1.} For all $i$, $\alpha_i=0$. Obviously, $N=K$ by the irreducibility of $K_{1i}$ and $D_{1i}\not=0$, $i=0, 1.$\\

 \noindent\textbf{Case 2.} There exists $k$ such that $\alpha_{k}\not=0.$
 If $k\not=1, \widetilde{1}$, for $\widetilde{k}\not=j\in\mathbf{I}_1$, we have:
 $$0\not=[y_jy_{\widetilde{k}}, w]\in N\cap K_{11}.$$
 Similar to Case 1, we have $N=K$.

 If $k=1$ or $\widetilde{1}$, then $w=y^{(2\varepsilon_1)}y_{\widetilde{1}}+\alpha_1y_1z+\alpha_{\widetilde{1}}y_{\widetilde{1}}z$. For $j\in\mathbf{I}_1$, we have:
 $$y_jy_1y_{\widetilde{1}}=[[y_jy_1, w], y^{(2\varepsilon_{\widetilde{1}})}]\in N\cap K_{10}.$$
  Similar to Case 1, we have $N=K$.

 Consequently,
$N=M'$ when  $D_{11}=0$ and  $N=M''$ when $D_{10}=0$.
\qed
\section{MGS of Type (II)}\label{002}

Let $L=W, S, H$ or $K$.
Recall $$\frak{V}^L=\{V\mid V \mbox{ is a nontrivial subspace of } L_{-1}\}.$$
 To  describe the MGS of type $(\textrm{II})$ of $L$, for any
 $V\in\frak{V}^L$, we define
$$\mathcal{M}(V)=\oplus_{i\geq-2}\mathcal{M}_{i}(V),$$
where
\begin{eqnarray}
&&\mathcal{M}_{-1}(V)=V;\ \ \mathcal {M}_{-2}(V)=[\mathcal {M}_{-1}(V), \mathcal {M}_{-1}(V)];\nonumber\\
 &&\mathcal{M}_{i}(V)=\{u\in L_{i}\mid [V, u]\subset
\mathcal{M}_{i-1}(V)\} \qquad \mbox{for}\; i\geq 0.\label{huam}
\end{eqnarray}

\begin{theorem}\label{theorem02llm}  Suppose $L=W$ or $S$.  All    MGS  of type  $\mathrm{(II)}$ of $L$ are characterized as follows:
\begin{itemize}
\item[$\mathrm{(1)}$] All   MGS of type  $\mathrm{(II)}$ of $L$ are precisely:
$$ \{\mathcal {M}(V)\mid V\in\frak{V}^L\}.$$
\item[$\mathrm{(2)}$] For any  $V$ and $V'$ in $\frak{V}^L$,
$$\mathcal{M}(V)
\stackrel{\mathrm{algebra}}{\cong}\mathcal{M}(V')\Longleftrightarrow V{\cong}V'.$$
\item[$\mathrm{(3)}$]  $L$ has exactly $(m+1)(n+1)-2$  isomorphism  classes of  MGS of type $\mathrm{(II)}$.
\item[$\mathrm{(4)}$] If $\mathrm{superdim}V=(k,l),$ then
\begin{align*}
 \mathrm{dim}\mathcal{M}(V)= \left\{\begin{array}{ll}
                                      2^{n}p^{m}(m+n)-2^{l}p^{k}(m+n-k-l),  & L=W; \\
                                      2^{n}p^{m}(m+n-1)+1-2^{l}p^{k}(m+n-k-l),  & L=S.
                                    \end{array}\right.
\end{align*}
 \end{itemize}
\end{theorem}
When $L=H$ or $K$, recall  definitions (\ref{12731eq1})-(\ref{eqsb2}) mentioned in Section 2. Put
\begin{align*}
   &\mathcal{V}^L= \left\{V \in\frak{V}^L, \mbox{ satisfying } J_3 \mbox{  is neither single nor twinned}\right\};\\
   &\mathcal{W}^{K}= \left\{V \in\frak{V}^K, \mbox{ satisfying } J_3 \mbox{  is not twinned}\right\}.
\end{align*}
   Suppose  $V\in\frak{V}^{K}$ is isotropic.  Put
$$\mathcal{M}^K(1, V)=\oplus_{i\geq -2}\mathcal{M}_i^K(1, V),$$
 where
 \begin{align*}
 &\mathcal{M}^K_{-2}(1, V)=\mathbb{F};\ \ \mathcal{M}^K_{-1}(1, V)=V;\\
&\mathcal{M}^K_{i}(1, V)=\{u\in {K}_i\mid [V, u]\subset \mathcal{M}^K_{i-1}(1, V), \ [1, u]\in\mathcal{M}^K_{i-2}(1, V)\}, \ i\geq 0.
\end{align*}

\begin{theorem}\label{mgs1276t1}  All    MGS  of type  $\mathrm{(II)}$ of $H$ and $K$   are characterized as follows:

For ${H}$,
\begin{itemize}
\item[$\mathrm{(1)}$]  All MGS  of type  $\mathrm{(II)}$ are precisely:
$$\left\{\mathcal{M}(V)\mid V\in\mathcal{V}^H\right\}.$$
\item[$\mathrm{(2)}$] For any $V$ and $ V'$ in $\mathcal{V}^H$,
$$\mathcal{M}(V)\stackrel{\mathrm{algebra}}{\cong}
\mathcal{M}(V')\Longleftrightarrow V{\cong}V'.$$
\item[$\mathrm{(3)}$] $H$ has exactly $\phi(r,n)$  isomorphism  classes of  MGS of type $\mathrm{(II)}$, where
\[
\phi(r,n)=\left\{\begin{array}{ll}
                       8^{-1}(r+1)(r(n+2)^2+2n^2+6-r)-2, & n \ \mbox{ is odd};\\
                       8^{-1}(r+1)(r(n+2)^2+2n^2+8)-2, & n \ \mbox{ is even}.
                    \end{array}
\right.
\]
\item[$\mathrm{(4)}$] For $V\in\mathcal{V}^H$,  if  $\beta$-$\mathrm{dim}V=(a, b, c, d),$ then
$$\mathrm{dim}\mathcal{M}(V)=p^{m}2^n+p^{2a}2^d-p^{a+b}2^{c+d}(m-2b+n-d-2c+1)-2.$$
\end{itemize}
For ${K}$,
\begin{itemize}
\item[$(1')$]   All MGS  of type  $\mathrm{(II)}$ are precisely:
\begin{align*}
&\left\{\mathcal{M}(V)\mid
V\in\mathcal{V}^K\mbox{ is neither nondegenerate  nor  isotropic}\right\} \\
\cup
&\left\{\mathcal{M}(V)\mid V\in\mathcal{W}^{K} \mbox{ is nondegenerate  or  isotropic}\right\}\\
\cup
 &\left\{\mathcal{M}^K(1,V)\mid
V\in\mathcal{V}^K\mbox{ is  isotropic}\right\}.
 \end{align*}
 \item[$(2')$] For all MGS  of type  $\mathrm{(II)}$,
\begin{align*}
& \mathcal{M}(V)\stackrel{\mathrm{algebra}}{\cong}
\mathcal{M}(V')\Longleftrightarrow V{\cong}V',\\
& \mathcal{M}^K(1, V)\stackrel{\mathrm{algebra}}{\cong}
\mathcal{M}^K(1, V')\Longleftrightarrow V{\cong}V'.
\end{align*}
\item[$(3')$] $K$ has exactly $\phi{(r, n)}$  isomorphism  classes of MGS  of type  $\mathrm{(II)}$, where
\[
\phi(r, n)=\left\{\begin{array}{ll}
                       8^{-1}(r+1)(r(n+2)^2+2n^2+4n+2-r)+r-1, & n \ \mbox{ is odd};\\
                       8^{-1}(r+1)(r(n+2)^2+2n^2+4n+8)+r-2, & n \ \mbox{ is even}.
                    \end{array}
\right.
\]
 \item[$(4')$] Let $\delta=1$ when $n-m-3=0\pmod{p}$  and $\delta=0$, otherwise.   Suppose
  $\beta$-$\mathrm{dim}V=(a, b, c, d).$
\begin{itemize}
\item[(a)] If $V$ is isotropic, then
\begin{align*}
&\mathrm{dim}\mathcal{M}(V)=p^{m}2^n-p^{b}2^{c}(m-2b+n-2c)-\delta, \mbox{ when }\ V\in\mathcal{W}^{K};\\
&\mathrm{dim}\mathcal{M}^K(1, V)=p^{m}2^n-p^{b+1}2^{c}(m-2b+n-2c)+p-\delta,  \mbox{ when }\ V\in\mathcal{V}^K.
\end{align*}
\item[(b)] If $V$ is not isotropic, then
$$\mathrm{dim}\mathcal{M}(V)= p^m2^n-(2r-2a+n-d)p^{2a+1}2^d-p,$$
when $V\in\mathcal{W}^{K}$  is nondegenerate satisfying  $ J_3$  is single;
$$ \mathrm{dim}\mathcal{M}(V)=p^{m}2^n+p^{2a+1}2^d-p^{a+b+1}2^{c+d}(m-2b+n-d-2c)-\delta,$$
when $V\in\mathcal{W}^{K}$  is nondegenerate satisfying  $ J_3$  is not single or $V\in\mathcal{V}^K$
is not nondegenerate.
\end{itemize}
\end{itemize}
\end{theorem}
\begin{lemma}\label{mgsl12731}  Suppose $L=W, S, H$ or $K$.
\begin{itemize}
\item[$\mathrm{(1)}$] $\mathcal{M}(V) $
is a $\mathbb{Z}$-graded subalgebra of $L$.   $\mathcal{M}^K(1, V)$ is a $\mathbb{Z}$-graded subalgebra of $K$.
\item[$\mathrm{(2)}$]
Suppose  $\Phi$ is a $\mathbb{Z}$-homogeneous automorphism of $L$. Then
$$\Phi(\mathcal{M}_i(V))=\mathcal{M}_i(\Phi(V))\ \mbox{ for all }\ i\geq -2.$$
 Moreover, $$\Phi(\mathcal{M}(V))=\mathcal{M}(\Phi(V)).$$
  For $K$,
$$\Phi(\mathcal{M}^K_i(1, V))=\mathcal{M}^K_i(1, \Phi(V)) \ \mbox{ for all }\ i\geq -2.$$
 Moreover,
$$\Phi(\mathcal{M}^K (1, V))=\mathcal{M}^K(1, \Phi(V)).$$
\item[$\mathrm{(3)}$]  If $M$ is an MGS of type $\mathrm{ (II)}$ of $L$,  then $M=\mathcal{M}(M_{-1})$ unless $L=K$, $M_{-1}$ is isotropic and $M_{-2}\not=0$. However,  in the latter case,  $M=\mathcal {M}^K(1, M_{-1})$.
\end{itemize}
\end{lemma}
\begin{proof} The approach is analogous to that used in the  case of modular Lie algebras \cite{HM}.
\end{proof}
\begin{remark}\label{mgsstandard1} Suppose $L=W, S, H$ or $K$.
 In view of Lemmas \ref{mgsl3} and \ref{mgsl12731}, for any  $V\in\frak{V}^L$, we may assume that $V$ is  standard [see (\ref{eqsb}), (\ref{eqsb2})].
\end{remark}
Now, we consider the case  $L=W$ or $S$. Suppose $V\in \frak{V}^L$ with $\mathrm{superdim}V=(k, l)$.
For $L=W$, it is easy to verify that
$\mathcal{M}_{0}(V)$ has a standard $\mathbb{F}$-basis
 $\mathcal{A}_{1}\cup\mathcal {A}_{2},$
where
\begin{eqnarray*}
&&\mathcal{A}_{1}=\{x_{i}\partial_{j}\mid i,j\in
\mathbf{I}(k,l)\},\\
&&\mathcal {A}_{2}=\{x_{i}\partial_{j}\mid
i\in\overline{\mathbf{I}}(k,l),j\in \mathbf{I}\}.
\end{eqnarray*}
Similarly, for $L=S$, $\mathcal{M}_{0}(V)$ has a standard  $\mathbb{F}$-basis $\mathcal{C}_{1}\cup
\mathcal {C}_{2}\cup \mathcal {C}_{3}$,  where
\begin{eqnarray*}
&&\mathcal
{C}_{1}=\{x_{i}\partial_{j}\mid i,j\in\mathbf{I}(k,l),i\neq j\},
\\
&&\mathcal {C}_{2}= \{x_{i}\partial_{j}\mid
i\in\overline{\mathbf{I}}(k,l),j\in\mathbf{I},i\neq j\},\\
&&\mathcal {C}_{3}=\{x_{1}\partial_{1}-(-1)^{|\partial_i|}x_{i}\partial_{i}\mid
i\in\mathbf{I}\backslash\{1\}\}.
\end{eqnarray*}
Moreover, in any case of $L=W$ or $S$,  $\mathcal{M}_{0}(V)$ has a standard co-basis in $W_{0}$:
\begin{equation}\label{liuliumelikyan311}\mathcal {A}_{3}=\{x_{i}\partial_{j}\mid i\in
\mathbf{I}(k,l),j\in \overline{\mathbf{I}}(k,l)\}.
\end{equation}

\begin{lemma}\label{yinlimax-noncontain} Suppose $U$, $V\in\frak{V}^L$, $L=W$ or $S$.
\begin{itemize}
\item[$\mathrm{(1)}$] $\mathcal{M}_{0}(V)$ is a maximal subalgebra of $L_{0}$.
%\item[$\mathrm{(2)}$] If $U\subsetneq V$  then $\mathcal{M}_{0}(U)\not\subset\mathcal{M}_{0}(V).$
\item[$\mathrm{(2)}$] $\mathcal{M}_{0}(U)=\mathcal{M}_{0}(V)$ if and only if $U=V.$
\end{itemize}
\end{lemma}
\begin{proof}
(1) Let $\mathfrak{G}_{0}$ be a subalgebra of $L_0$ which strictly contains $\mathcal{M}_{0}(V)$. It is clear that $\mathfrak{G}_{0}$   contains a nonzero element of form
$
B=\sum_{h,t\geq1}
\alpha_{ht}x_{i_{h}}\partial_{j_{t}},$
where $0\neq\alpha_{ht}\in\mathbb{F},i_{h}\in\mathbf{I}(k,l),
j_{t}\in\overline{\mathbf{I}}(k,l)$.

When $L=W$,  for any  $i\in \mathbf{I}(k,l)$ and $j\in \overline{\mathbf{I}}(k,l),$ one has
$x_{i}\partial_{i_{1}}\in\mathcal {A}_{1}$ and
$x_{j_{1}}\partial_{j}\in\mathcal {A}_{2}$. Then
$$
x_i\partial_j=\alpha_{11}^{-1}[x_{i}\partial_{i_{1}},[B,x_{j_{1}}\partial_{j}]]\in \mathfrak{G}_{0},
$$
 showing that the co-basis $\mathcal {A}_{3}\subset\mathfrak{G}_{0}$. Hence $\mathfrak{G}_{0}=W_0$.

When $L=S$, suppose $|\mathbf{I}(k,l)|>1$ and $|\overline{\mathbf{I}}(k,l)|>1$.
Choosing
$x_{j_{1}}\partial_{j}$ in $\mathcal {C}_{2}$  with  $j\in
\overline{\mathbf{I}}(k,l)\backslash \{j_{1}\}$ and  $x_{i}\partial_{i_{1}}$ in $\mathcal {C}_{1}$
with $i\in \mathbf{I}(k,l)\backslash \{i_{1}\}$, we have
$$
x_i\partial_j=[x_{i}\partial_{i_{1}},[B,x_{j_{1}}\partial_{j}]]\in \mathfrak{G}_{0},
$$
 showing that the co-basis $\mathcal {A}_{3}\subset\mathfrak{G}_{0}$ and then $\mathfrak{G}_{0}=S_0$.
For the remaining case $|\mathbf{I}(k,l)|=1$ or  $|\overline{\mathbf{I}}(k,l)|=1$, the argument is similar and much easier.

% (2) As in (1), we may assume that $U$ has a standard basis extending to a standard one of $V$. The rest  of the argument is similar to the proof of (1).

 (2) One direction is obvious. Note that  one may choose bases  of $U$ and $V$ as follows:
 $$
 \overbrace{E_1,\ldots, E_r}^{\textrm{cobasis in}\; U},
  \overbrace{F_1,\ldots, F_s}^{\textrm{basis of}\; U\cap V},
  \overbrace{G_1,\ldots, G_t}^{\textrm{cobasis in}\;V}
 $$
 where $(E_1,\ldots, E_r,  F_1,\ldots, F_s,  G_1,\ldots, G_t)$ is a permutation of   $\partial_i$'s.
 Keeping in mind the standard co-basis  (\ref{liuliumelikyan311}), we are done by a similar argument as in (1).
\end{proof}

\begin{proposition}\label{prowmax} $\mathcal{M}(V)$ is maximal in $L$ for any $V\in\frak{V}^L$, $L=W$ or $S$.
\end{proposition}
\begin{proof}
 Let $M$ be an MGS containing $\mathcal{M}(V)$. Then $\mathcal{M}_{i}(V)\subset M_i$ for all $i\geq -1.$
In particular, because of the maximality of $\mathcal{M}_{0}(V),$ it must be $M_0=\mathcal{M}_{0}(V)$ or $M_0=L_0.$\\

\noindent\textbf{Case 1.} Suppose $M_0=\mathcal{M}_{0}(V)$. By induction, it is routine to verify that  $M_i=\mathcal{M}_i(V)$ for all
$i\geq 0.$ Assume on the contrary that $M$ strictly contains  $\mathcal{M}(V)$. Then $M_{-1}\supsetneq \mathcal{M}_{-1}(V)=V.$
 Note that $\mathcal{M}_{0}(V)= M_{0}=\mathcal{M}_{0}(M_{-1})$ from Lemma \ref{mgsl12731}(3).  Thus,
Lemma \ref{yinlimax-noncontain}(2) forces $M_{-1}=L_{-1} $. Pick any $i\in \mathbf{I}(k,l)$, $j\in \overline{\mathbf{I}}(k,l)$ and any $h\neq i,j.$ We are able to check that
 $$A=(-1)^{|x_h|}x_ix_j\partial_j-(-1)^{|x_j|}x_ix_h\partial_h\in S_1\subset W_1.$$ Moreover,
$A\in \mathcal{M}_1(V)=M_1.$ Since $M_{-1}=W_{-1} =S_{-1}$, we have
$$x_i\partial_j=(-1)^{(|x_h|+|x_i||x_j|)}[\partial_j, A]\in M_0=\mathcal{M}_0(V).$$  This contradicts
the fact  that $x_i\partial_j\in\mathcal{A}_3$  [see (\ref{liuliumelikyan311})]. Therefore, $M=\mathcal{M}(V).$\\

\noindent\textbf{Case 2.} Suppose $M_0=L_0$. In this case, since $L_{-1}$ is irreducible as $L_0$-module, we have $M_{-1}=L_{-1}.$ Hence $M$ is an MGS of type (I). By Theorem \ref{thm01}(1) and (2),
$M_1=W_1'$, $W_1'',$ $S_1''$, or $\{0\}.$ In Case 1, we have shown that $A\in \mathcal{M}_1(V)$. However, it is
clear that $A$ does not belong to $W_1'$, $W_1'',$ $S^{''}$, or $\{0\}.$  Hence $\mathcal{M}_1(V)\not\subset M_1.$ This contradicts
the assumption that $M$ is a graded subalgebra containing $\mathcal{M}(V).$
\end{proof}

\noindent
\textbf{Proof of Theorem \ref{theorem02llm}}
(1), (2) and (3) are  immediate consequences  of  Lemmas \ref{mgsl3},  \ref{mgsl12731}(3) and  Proposition \ref{prowmax}.
It remains to show the dimension formulas.
 For $W$,  $\mathcal{M}(V)$ has a standard $\mathbb{F}$-basis
which is a disjoint union:
\begin{align*}
&\{x^{(\alpha)}x^{u} \partial_{i}\mid   \alpha\in
\mathbf{A}(m), u\in \mathbf{B}(n);\;
i\in\mathbf{I}(k,l)\}\\
\cup
&\{x^{(\alpha)}x^u
\partial_{i}\mid \;
i\in
\overline{\mathbf{I}}(k,l) \;\mbox{and}\; \exists j\in \overline{\mathbf{I}}(k,l)\;\mbox{such that}\; \partial_{j}(x^{(\alpha)}x^u)\neq 0\}.
\end{align*}
A standard and direct computation shows that:
$$\mathrm{dim}\mathcal{M}(V)=2^{n}p^{m}(m+n)-2^{l}p^{k}(m+n-k-l).$$
Similarly, for $S$, we have:
$$\mathrm{dim}\mathcal{M}(V)=2^{n}p^{m}(m+n-1)+1-2^{l}p^{k}(m+n-k-l).$$
   \qed

Next, we consider the case  $L=H$ or $K$. In this case, we shall frequently use the standard facts mentioned in Remark \ref{mgsr2} without notice. Suppose $V\in\frak{V}^L$ with $\beta$-$\dim V=(a,b,c,d)$.
In order to prove Theorem \ref{mgs1276t1} we list the following assertions. For simplicity, we write $\lambda_{i, j}$ for a nonzero element in $\mathbb{F}$, where $i, j\in\mathbf{I}$. Recall definitions (\ref{12731eq1})-(\ref{mgseq12734}).
Put  \begin{align}
&\mathcal{V}_{\frak{i}}^L=\left\{V\in\frak{V}^L\mid V \mbox{  is  isotropic  and } J_3 \mbox{ is not twinned}\right\};\nonumber\\
& \mathcal{V}_{\frak{n}}^L=\left\{V\in\frak{V}^L\mid V \mbox{ is nondegenerate and } J_i \mbox{ is not twinned, } i=1, 3\right\};\nonumber\\
&\mathcal{V}_{\frak{d}}^L=\left\{V\in\frak{V}^L\mid V \mbox{ is degenerate}, J_3 \mbox{ is empty and } J_1 \mbox{ is not twinned} \right\}.\label{mgs130407notion}
\end{align}
\begin{lemma}\label{mgs1273l2} For $H$,  put
$$A_i=\mathrm{span}_{\mathbb{F}}\{u\in H_i\mid \nu(u)=0, 1 \mbox{ or } (\frac13)^{k}2^l, \ k, l\in\mathbb{N}, \ l>1\}.$$ Then
\begin{itemize}
\item[$(1)$]   $A_i=\mathcal{M}_i(V)$, $i\geq -1$.
\item[$(2)$] The subalgebra $A_0=\mathcal{M}_0(V)$ is maximal in $H_0$  if and only if
$$V\in\mathcal{V}_{\frak{i}}^H\cup\mathcal{V}_{\frak{n}}^H\cup\mathcal{V}_{\frak{d}}^H.$$
\end{itemize}
\end{lemma}
\begin{proof}
(1) It follows by using induction on $i$, $i\geq-1$.

(2) Obviously, the torus $T$ mentioned in Remark \ref{mgsr2}(2) is contained in $\mathcal{M}_0(V)$.
For any $h\in H_0$ and $h\not\in \mathcal{M}_0(V)$, put $\overline{M}=\mathrm{alg}(\mathcal{M}_0(V)+\mathbb{F}h)$.
Firstly, we show the maximality of $\mathcal{M}_0(V).$ It suffices to prove $H_0=\overline{M}$.\\

\noindent\textbf{Case 1.} $V\in\mathcal{V}_{\frak{i}}^H$. Notice that
\begin{equation*}
\nu(y_i)=\left\{\begin{array}{llll}
  0 & i\in J_2; \\
  \frac{1}3  & i\in \bar{J}_2; \\
  2 & i\in J_3
\end{array}\right.\ \
\mbox{ and }\ \
\mathcal{M}_0(V)=\mathrm{span}_{\mathbb{F}}\{y_iy_j\mid (i, j)\in J_2\times\mathbf{ I}\cup J_3\times J_3\}.
\end{equation*}
We may assume that $h$ is a monomial  with $\nu(h)=\frac{1}9$ or $ \frac{2}3$.
When $h=y_iy_j$, $(i,j)\in \bar{J}_2\times \bar{J}_2$, we have:
\begin{align*}
&y_ky_l=\lambda_{k,l}[[y_iy_j, y_{\widetilde{j}}y_k], y_{\widetilde{i}}y_l]\in\overline{M} \ \mbox { for all } \ (k, l)\in \bar{J}_2\times \bar{J}_2,\\
&y_ky_s=\lambda_{k,s}[y_ky_l, y_{\widetilde{l}}y_s]\in\overline{M}\ \mbox { for all } \ s\in {J}_3.
\end{align*}
 Thus,  $H_0=\overline{M}$.
When $h=y_iy_j$, $(i,j)\in \bar{J}_2\times {J}_3$, if $j\in I_{03}$ or $I_{03}$ is not empty, we get $H_0=\overline{M}$ in an analogous way as above.
%\begin{align*}
%&y_ay_l=\lambda_{ab}[[y_iy_j, y_{\widetilde{j}}y_k], y_{\widetilde{i}}y_l]\in\overline{M}, \ \ \mbox { for all } \ (k, l)\in {J}_3\times \bar{J}_2.\\
%&y_ay_s=\lambda_{ac}[y_ay_l, y_{\widetilde{l}}y_s]\in\overline{M}, \ \ \mbox { for all } \ s\in {\overline{J}}_2,
%\end{align*}
%where $\lambda_{ab}, \lambda_{ac}\not=0$. Thus,  $H_0=\overline{M}$.
Otherwise, we may assume that $I_{03}$ is empty.  If $J_3$ is single, we have:
\begin{align*}
&y_ky_{m+n-d}=\lambda_{k,m+n-d}[y_iy_{m+n-d}, y_{\widetilde{i}}y_k]\in\overline{M} \ \mbox { for all } \ k\in \bar{J}_2,\\
&y_ky_{l}=\lambda_{k,l}[y_ky_{m+n-d}, y_{m+n-d}y_l]\in\overline{M}\ \mbox { for all } \ l\in \bar{J}_2.
\end{align*} It follows that $H_0=\overline{M}$.
If $J_3$ is neither single nor twinned, for $s\in J_{13}$, $s\not=\widetilde{j}$, we have:
\begin{align*}
&y_ky_s=\lambda_{k,s}[[y_iy_j, y_{\widetilde{i}}y_k], y_{\widetilde{j}}y_s]\in\overline{M}\ \mbox { for all } \ k\in \bar{J}_2,\\
&y_ky_{\widetilde{j}}=\lambda_{k,\widetilde{j}}[y_ky_s, y_{\widetilde{s}}y_{\widetilde{j}}]\in\overline{M} \ \ s\not=j \mbox{ and } \widetilde{j},\\
&y_ky_l=\lambda_{k,l}[y_jy_k, y_{\widetilde{j}}y_l]\in\overline{M}\ \mbox { for all } l\in \bar{J}_2.
\end{align*}
Thus,  $H_0=\overline{M}$.\\

 \noindent\textbf{Case 2.} $V\in\mathcal{V}_{\frak{n}}^H$. Notice that
\begin{equation*}
\nu(y_i)=\left\{\begin{array}{llll}
 1 & i\in J_1; \\
  2 & i\in J_3
\end{array}\right.\ \
\mbox{ and }\ \
\mathcal{M}_0(V)=\mathrm{span}_{\mathbb{F}}\{y_iy_j\mid (i, j)\in J_1\times J_1\cup J_3\times J_3\}.
\end{equation*}
We may assume that $h$
is a linear combination of monomials  with value 2.
When $h=y_iy_j$, $(i,j)\in (I_{01}\cup\bar{I}_{01})\times {J}_3$, using the same method as in  Case 1, we get $H_0=\overline{M}$.
When  $h=\sum_{i\in I_{11}}a_iy_ky_i$, where $k\in J_{3}$, $a_i\in \mathbb{F}$, $a_j\not=0$,
we get $H_0=\overline{M}$  if $I_{01}$ is not empty or  $J_{1}$ is single by a similar argument as in Case 1.
%we consider the following conditions.
%If $I_{01}$ is not empty,
% we have
% $$y_ky_l=(a_j)^{-1}\lambda_{k,l}[h, y_jy_l]\in \overline{M}\ \mbox{ for  }\ l\in I_{01}\cup\bar{I}_{01}.$$
Thus, it suffices to  consider the condition  that  $I_{01}$ is  empty and  $J_{1}$ is neither single nor twinned.
%If $J_{1}$ is single, similar to  case (a), $H_0=\overline{M}$ follows.
For distinct $l, s, j\in I_{11}$,  we have
\begin{align*}
&y_ky_s=(a_j)^{-1}\lambda_{k,s}[y_ly_s, [y_jy_l, h]]\in \overline{M},\\
&y_ey_k=\lambda_{e,k}[y_ky_s, y_sy_e]\in\overline{M}\  \mbox{ for any }\ s\not=e\in I_{11}.
\end{align*}
For any $i\in I_{11}$, $f\in I_{03}$ and $t\in I_{13}$, we have
$$y_fy_i=\lambda_{f,i}[y_ky_i, y_{\widetilde{k}}y_f]\in\overline{M}\ \mbox{ and }\ y_ty_i=\lambda_{t,i}[y_fy_i, y_{\widetilde{f}}y_t]\in\overline{M}.$$
Thus,  $H_0=\overline{M}.$ \\

 \noindent\textbf{Case 3.} $V\in\mathcal{V}_{\frak{d}}^H$. Notice that
\begin{equation*}
\nu(y_i)=\left\{\begin{array}{llll}
 1 & i\in J_1; \\
 0 & i\in J_2; \\
  \frac{1}3  & i\in \bar{J}_2
\end{array}\right.\ \
\mbox{ and }\ \
\mathcal{M}_0(V)=\mathrm{span}_{\mathbb{F}}\{y_iy_j\mid (i, j)\in {J}_2\times \mathbf{I}\cup J_1\times {J}_1\}.
\end{equation*}
We have $H_0=\overline{M}$ by the same method  as in  Cases 1 and 2.

Conversely,  we consider the co-basis of $\mathcal{M}_0(V)$ in $H_0$:
 \begin{equation}\label{co-basis}\{y_iy_j\mid (i, j)\in {J}_1\times \bar{J}_2\cup J_1\times {J}_3\cup \bar{J}_2\times \bar{J}_2\cup \bar{J}_2\times {J}_3\}.\end{equation}
Notice that if $V\not\in\mathcal{V}_{\frak{i}}^H\cup\mathcal{V}_{\frak{n}}^H\cup\mathcal{V}_{\frak{d}}^H$ , then $V\in\frak{V}^H$ must satisfy one of the following conditions:\\

 (i) None of $J_1, J_2, J_3$ is empty. In this case, we choose a monomial $h$ of $H_0$ with $\nu(h)=2$. Then
 there do not exist monomials with value $\frac{1}3$ in
 $\overline{M}$.\\

 (ii) $J_3$ is twinned, i.e., $J_3=\{j, {\widetilde{j}}\}$, where $j\not=\widetilde{j}\in\mathbf{I}_1$. In this case, let $h=y_iy_j$, where
 \[
 i\in\left\{\begin{array}{ll}
              \bar{J}_2, & J_1\ \mbox{ is empty}; \\
              {J}_1, & \ \mbox{ otherwise}.
            \end{array}
            \right.
  \]
 Then
 $y_iy_{\widetilde{j}}\not\in\overline{M}$.\\

  (iii) $J_1$ is twinned, i.e., $J_1=\{m+n-1, m+n\}$. In this case, let $h=y_k(y_{m+n-1}+\sqrt{-1}y_{m+n})$, where
 \[
 k\in\left\{\begin{array}{ll}
              \bar{J}_2, & J_3\ \mbox{ is empty}; \\
              {J}_3, & \ \mbox{ otherwise}.
            \end{array}
            \right.
  \]
 Then
  $y_ky_{m+n}\not\in\overline{M}$.

 % $J_2$ is not empty and $J_1$ is twinned, i.e., $J_1=\{m+n-1, m+m\}$. Let $h=y_k(y_{m+n-1}+\sqrt{-1}y_{m+n})$, $k\in \bar{J}_2$, then
 %$y_ky_{m+n}\not\in\overline{M}$.

 %(5)  $J_3$ is not empty and $J_1$ is twinned, i.e., $J_1=\{m+n-1, m+m\}$. Let $h=y_k(y_{m+n-1}+\sqrt{-1}y_{m+n})$, $k\in {J}_3$, then
% $y_ky_{m+n}\not\in\overline{M}$.

 Therefore, $\overline{M}$ is a nontrivial  subalgebra of $H_0$ strictly containing $\mathcal{M}_0(V)$
 when (i), (ii) or (iii) holds, which implies that $\mathcal{M}_0(V)$
 is not a maximal subalgebra of $H_0$.
 %The proof is complete.
\end{proof}

\begin{proposition}\label{mgs1276p1}
The subalgebra $\mathcal{M}(V)$ is maximal in $H$ if and only if $V\in\mathcal{V}^H$.
\end{proposition}
\begin{proof}
%From Lemma \ref{mgs1273l2}, we have $\mathcal{M}_1(V)=\mathrm{span}_{\mathbb{F}}\{u\in H_1\mid \nu(u)=1, 0, 8, 4, \frac{4}3\}.$

If $J_3=\{m+n-d\}$, from Lemma \ref{mgs1273l2}(1), we know that
$$\mathcal{M}_i(V)=\mathrm{span}_{\mathbb{F}}\{u\in H_i\mid \nu(u)=1, 0\}, \ i\geq -1,$$
which implies that
$$\mathrm{alg}(\mathcal{M}(V)+\mathbb{F}y_{m+n-d})\subset\mathrm{span}_{\mathbb{F}}\{u\in H\mid \nu(u)=1, 0\}+\mathbb{F}y_{m+n-d}.$$
 Thus,
$y_i, y_jy_{m+n-d}\not\in\mathrm{alg}(\mathcal{M}(V)+\mathbb{F}y_{m+n-d})$ if  $\nu(y_i)={\frac13}$  or   $\nu(y_j)=1,$ which contradicts the maximality of $\mathcal{M}(V)$.

If $J_3=\{j, {\widetilde{j}}\}$, where $ j\not=\widetilde{j}\in\mathbf{I}_1$, from Lemma \ref{mgs1273l2}(1),
we know that
$$\mathcal{M}_i(V)=\mathrm{span}_{\mathbb{F}}\{u\in H_i\mid \nu(u)=1, 0, (\frac{1}3)^{k}4\}, \ i\geq 0.$$
Then  for any monomial  $u\in \mathcal{M}_i(V)$,   we have:
\[
[y_{j}, u]=0\ \ \mbox { or } \ \ [y_{j}, u]=wy_j,
\] where  $0\not=w\in H_{i-1}$ with $D_{\widetilde{j}}(w)=0$,
%$[y_{\widetilde{j}}, [{y_j, u}]]=0$
 which implies that
 $$y_{\widetilde{j}} \not\in \mathrm{alg}(\mathcal{M}(V)+\mathbb{F}y_{j}).$$
This contradicts the maximality of $\mathcal{M}(V)$.

Conversely, let us prove the maximality of $\mathcal{M}(V)$.
By  definition (\ref{huam}), it is sufficient  to show  that $\overline{M}=\mathrm{alg}(\mathcal{M}(V)+\mathbb{F}h)=H$,
where $h=y_i$, $i\in \bar{J}_2\cup J_3$.
Note that $\mathcal{M}_{1}(V)\not=0$ for $|\mathbf{I}_0|\geq2$.
From Lemmas \ref{dibu}(2) and \ref{mgsl2}(1), it suffices to prove $H_{-1}, H_0\subset\overline{M}$.
For $V\in\mathcal{V}^H$, we discuss the following cases:\\

\noindent\textbf{Case 1.}  $J_2$ is  not empty.
When $i\in \bar{J}_2$, since
$$y_{\widetilde{i}}\in V\ \mbox{ and }\ y_j=\lambda_{i, j}[y_i, y_{\widetilde{i}}y_j]\in \overline{M} \ \mbox{ for }
\ \widetilde{i}\not= j\in\mathbf{I},$$ we have $H_{-1}\subset\overline{M}$.
 When  $i\in {J}_3$, for all $j\in J_3$ with $j\not=i, \widetilde{i}$, we have:
\[y_j=\lambda_{i,j}[y_i, y_{\widetilde{i}}y_j]\ \ \mbox { and }\ \ y_{\widetilde{i}}=\lambda_{\widetilde{i},j}[y_j, y_{\widetilde{i}}y_{\widetilde{j}}].\]
Note that
$$y_l=\lambda_{l,i}[y_i,[y_{\widetilde{i}}, y_iy_{\widetilde{i}}y_l]]\in \overline{M}\ \mbox{ for all }\ l\in \bar{J}_2.$$
Thus we have
$H_{-1}\subset\overline{M}$.
Note that for an arbitrary monomial $u\in H_0$, there exists $k\in \mathbf{I}$ such that $uy_k\not=0$ and $\nu(uy_k)=0$.  Then we have
$$u=\lambda_{\widetilde{k}, k}[y_{\widetilde{{k}}}, uy_k]\in\overline{M},$$
 which implies that $H_0\subset\overline{M}$.
Thus, we have $\overline{M}=H$.\\

\noindent\textbf{Case 2.}  $J_2$ is empty. Obviously,  $J_1$ and $J_3$ are not empty.
Then we have $H_{-1}, H_0\subset\overline{M}$ by the same method as in Case 1.
\end{proof}
%\begin{remark}\label{mgs1276r1} Keep  notions as above.
%\begin{itemize}
%\item[(1)] If $J_3$ is  single and $V$ is nondegenerate, $\mathcal{M}(V)+\mathbb{F}y_{2r+2q+1}$
%is a maximal graded subalgebra of $H$ with property $(\mathrm{III})$.
%\item[(2)] If $J_3$ is  single and $V$ is degenerate, $\mathcal{M}(V)+\mathbb{F}y_{2r+2q+1}\cong \mathcal{M}(V')$, which
%is a maximal graded subalgebra of $H$ satisfying $f$-$\dim V=(k, s, t, l+1)$, $s\not=k$ or $t\not=0$.
%\item[(3)] If $J_3$ is  twinned, for $i\in J_3$,  $\mathcal{M}(V)+\mathbb{F}y_{i}\cong \mathcal{M}(V')$, which
%is a maximal graded subalgebra of $H$ satisfying $f$-$\dim V=(k, s, t+1, l)$.
%\end{itemize}
%\end{remark}

%$V$ is called standard, if $V=\mathrm{span}_{\mathbb{F}}\{y_i\mid i\in X\}$, where $X$ is a nontrivial subset of $\mathbf{I}$. Moreover,
%if $V$ is isotropic, put $V^{\bot}=\mathrm{span}_{\mathbb{F}}\{y_i\mid i\in \mathbf{I}\backslash X\}$; if  $V$ is nondegenerate,
% and the rank of $\beta$ on $V_{\bar{1}}$ is $d$,
%put $V^{\bot}=\mathrm{span}_{\mathbb{F}}\{x_i\mid i\in \mathbf{I}\backslash X\}$.

To avoid confusion, we rewrite $\mathcal{M}^L_i(V)$ for $\mathcal{M}_i(V)$,  $\mathcal{M}^L(V)$ for $\mathcal{M}(V)$, $L=H$ or $K$.
\begin{lemma}\label{mgs1277l1} Let $\gamma=\lfloor\frac{i+2}2\rfloor$ for $i>0$. Put
\[
\widetilde{\mathcal{M}}_{j}(V)=\left\{\begin{array}{ll}
                                        0, & j>\eta-2; \\
                                        \mathcal{M}^H_{j}(V), & j<\eta-2;
 \end{array}
\right.
\
\widetilde{\mathcal{M}}_{\eta-2}(V)=\left\{\begin{array}{ll}
                                        0, & V\mbox{ is nondegenerate} \\
                                        & \mbox{and}\ J_3\ \mbox{is single};\\
                            \mathbb{F}y^{(\pi)}y^{\omega}, & \mbox{otherwise, }
 \end{array}
\right.
\]
where $\pi=(p-1, \ldots, p-1)\in\mathbb{N}^{2r}$, $\eta=2r(p-1)+n$ and $\omega=\langle m+1, \ldots, m+n\rangle$. Then
\begin{itemize}
\item[$\mathrm{(1)}$] $\mathcal{M}^K_0(V)=\mathcal{M}^H_0(V)\oplus\mathbb{F}z$.
\item[$\mathrm{(2)}$] If $V$ is not isotropic, for $i>0$,
$$\mathcal{M}^K_i(V)=\widetilde{\mathcal{M}}_i(V)\oplus
\widetilde{\mathcal{M}}_{i-2}(V)z\oplus\cdots\oplus\widetilde{\mathcal{M}}_{i-2\gamma}(V)z^\gamma.$$
\item[$\mathrm{(3)}$] If $V$ is isotropic, for $i>0$,
$$\mathcal{M}^K_i(V)=\widetilde{\mathcal{M}}_i(V)\oplus
\overline{H}_{i-2}z\oplus\cdots\oplus \overline{H}_{i-2\gamma}z^\gamma.$$
\end{itemize}
\end{lemma}
\begin{proof}
(1)  It is obvious.

(2) Use induction on $i$. Clearly, $\widetilde{\mathcal{M}}_i(V)\subset\mathcal{M}^K_i(V)$.

$``\supset"$:
For $gz^k\in\widetilde{\mathcal{M}}_{i-2k}(V)z^k$,  $0<k\leq \gamma$, we know that
$$[y_l, gz^k]=[y_l, g]z^k+y_lgz^{k-1}.$$
Note that
$$y_lg\in \widetilde{\mathcal{M}}_{i-2(k-1)-1}(V)\ \mbox{ for } \nu(y_l)=1, 0.$$
 By induction on $\mathrm{zd}(g)$,
we have:
$$y_lgz^{k-1},\ [y_l, g]z^k\in \mathcal{M}^K_{i-1}(V).$$
 Thus, $gz^k\in \mathcal{M}^K_i(V)$.

$``\subset"$: For any $u\in \mathcal{M}^K_i(V)$, by Lemma \ref{mgsl2}(2), we may assume that
$$u=u_i+u_{i-2}z+\cdots+u_{i-2\gamma}z^\gamma,$$
 where $u_j\in \overline{H}_j$ for $i-2\gamma\leq j \leq i$.
Note that $\mathcal{M}^K_{-2}(V)\not=0$, since $V$ is not isotropic. Then we have:
$$u_{i-2}+u_{i-4}z+\cdots+u_{i-2\gamma}z^{\gamma-1}=2^{-1}[1, u]\in \mathcal{M}^K_{i-2}(V).$$
By induction, we have $u_{j}\in \widetilde{\mathcal{M}}_{j}(V)$ for $i-2\gamma\leq j \leq i-2$.
Moreover,
$$u_{i-2}z+u_{i-4}z^2+\cdots+u_{i-2\gamma}z^{\gamma}\in \mathcal{M}^K_{i}(V).$$
Consequently, $u_i\in \widetilde{\mathcal{M}}_{i}(V)$.

(3) When  $V$ is isotropic, note that  $\nu(y_k)=0$ for all $y_k\in V$. The remaining  discussion is analogous to that of the condition (2).
\end{proof}
\begin{proposition}\label{mgs1277p3}
The subalgebra $\mathcal{M}^K(V)$ is maximal in $K$ if and only if
 $V\in \mathcal{V}^K$ when $V$ is neither nondegenerate    nor isotropic; $V\in\mathcal{W}^{K}$, otherwise.
\end{proposition}
\begin{proof}
The proof of the necessity is similar to the one of  Proposition \ref{mgs1276p1}. We only consider the sufficiency.
For any $u\in K$, $u\not\in \mathcal{M}^K(V)$, put $\overline{M}=\mathrm{alg}(\mathcal{M}^K(V)+\mathbb{F}u)$.
Then there exist $v_1,\ldots, v_i\in V$ such that
$$0\not=h=[v_1,[\cdots,[v_i, u]\cdots]]\in\overline{M}\cap K_{-1}\  \mbox{ and  }\ h\not\in V.$$
When $J_3$ is neither single nor twinned, by Proposition \ref{mgs1276p1},
we have $ H\subset\overline{M}.$
When $J_3$ is single and $V$ is isotropic, we may assume that $h=y_i$, $i\in\bar{J}_2\cup\{m+n-d\}$.
If $i\in\bar{J}_2$, for any $j, k\in\bar{J}_2$, we  obtain that
$$y_{m+n-d}=\lambda_{i, \widetilde{i}}[y_{\widetilde{i}}, [y_i, y_{m+n-d}z]], \ \
y_{m+n-d}y_jy_k=\lambda_{j, k}[y_{m+n-d}, y_jy_kz]$$
are in $\overline{M}$ from Lemma \ref{mgs1277l1}(3).
Moreover,
$$
y_jy_k=-[y_{m+n-d}, y_{m+n-d}y_jy_k],\ \ y_{m+n-d}y_k=\lambda_{m+n-d, k}[y_{\widetilde{j}}, y_{m+n-d}y_jy_k] $$
are in $\overline{M}$. Keeping in mind the co-basis (\ref{co-basis}),
 we have  $H_0\subset\overline{M}$, which also holds
when $J_3$ is single and $V$ is nondegenerate.
From  the irreducibility of $K_{-1}$, $K_{10}$ and  $K_{11}$, as well as Lemma \ref{mgs1277l1}, we obtain that
 $\mathcal{M}^K(V)$ is maximal in $K$.
\end{proof}

In the same way as in  Proposition \ref{mgs1277p3}, one may check the following proposition.
\begin{proposition}\label{mgs1282p1}
If  $V$ is isotropic, the subalgebra $\mathcal{M}^K(1, V)$ is maximal in $K$ if and only if $V\in\mathcal{V}^K$.
\end{proposition}
\begin{convention}\label{mgsconvention3} For simplicity,
put $\mathcal{O}_{X}=\mathrm{span}_{\mathbb{F}}\{y_{i_1}\cdots y_{i_s}\mid i_1, \ldots, i_s\in X, s\geq 1\}$,
$\mathcal{Q}_X=\mathrm{span}_{\mathbb{F}}\{y_{i}\mid i\in X\}$ and $\mathcal{Y}^+_{X}=\mathcal{Y}_{X}\oplus\mathbb{F}\cdot1$,
where $X$ is a subset of $\mathbf{I}$ and $\mathcal{Y}=\mathcal{O} $ or $\mathcal{Q}$.
\end{convention}
\noindent
\textbf{Proof of Theorem \ref{mgs1276t1}}
For (1) and $ (1')$, the proofs  follow  from Lemma \ref{mgsl12731}(3), Propositions \ref{mgs1276p1}, \ref{mgs1277p3}
and \ref{mgs1282p1}.

For $H$, from Lemma \ref{mgs1273l2}(1) and (1), we obtain that
$$\dim{\mathcal{M}^H(V)}=\dim H-
\dim(\mathcal{O}^+_{J_1}\mathcal{O}_{\bar{J}_2}\oplus\mathcal{O}^+_{J_1\cup\bar{J}_2}\mathcal{Q}_{{J}_3}).$$

For $K$, from Lemma \ref{mgs1277l1} and $(1')$, we obtain that
\begin{align*}
 &\dim{\mathcal{M}^K(1, V)}=p(\dim\mathcal{M}^H(V)+2);\\
 &\dim{\mathcal{M}^K(V)}=\left\{\begin{array}{ll}
                          \dim\mathcal{M}^H(V)+1+(p-1)(\dim H+2), & V\in\mathcal{W}^{K} \mbox{  is isotropic}; \\
                          p(\dim\mathcal{M}^H(V)+2), &  \mbox{  otherwise}.
                        \end{array}
\right.
\end{align*}
By a standard and direct computation we get the formulas (4) and $(4')$.  Noting that $\dim\mathcal{M}_0(V)=\dim\mathcal{M}_0(V')$ if $\mathcal{M}(V) \cong$$
\mathcal{M}(V')$  and using the same method as in Theorem \ref{theorem02llm}(2),   (2) and $(2')$ hold. From (1), $(1')$ and (2), $(2')$,  we obtain  that (3) and $(3')$ hold.\qed

%%%%%%%%%%%%%%%%%%%%%%%%%%%%%%%%%%%%%%%%%%%%%%%%%%%%%%%%%%%%%%%%%%%
%%%%%%%%%%%%%%%%%%%%%%%%%%%%%%%%%%%%%%%%%%%%%%%%%%%%%%%%%%%%%%%%%%%%%%5
%%%%%%%%%%%%%%%%%%%%%%%%%%%%%%%%%%%%%%%%%%%%%%%%%%%%%%%%%%%%%%%%%%%%%%5
\section{MGS of Type (III)}\label{003}
Suppose $L=W, S, H$ or $K$. Recall that an MGS of type (III) of $L$, $M$,  satisfies the condition
\begin{equation*}\label{yubeijiiuhao}
 M_{-1}=L_{-1}\quad \mbox{and}\quad  M_{0}\neq L_{0}.
\end{equation*}
Let $\mathfrak{G}_{0}$ be a nontrivial subalgebra of $L_{0}$.
 Define a graded subspace of $L$ as follows:
\begin{equation*}
\mathcal {M}(L_{-1},\mathfrak{G}_{0})=\oplus_{i\geq -2}\mathcal
{M}_{i}(L_{-1},\mathfrak{G}_{0}),
\end{equation*}
where
\begin{eqnarray}
&&\mathcal {M}_{-i}(L_{-1},\mathfrak{G}_{0})=L_{-i},\ i<0; \quad\mathcal
{M}_{0}(L_{-1},\mathfrak{G}_{0})=\mathfrak{G}_{0};
\nonumber\\
&&
\mathcal {M}_{i}(L_{-1},\mathfrak{G}_{0})=\{u\in L_{i}\mid
[L_{-1},u]\subset\mathcal
{M}_{i-1}(L_{-1},\mathfrak{G}_{0})\}\ \mbox{ for } i>0.\label{m33}
\end{eqnarray}
It is easy to see that $\mathcal
{M}(L_{-1},\mathfrak{G}_{0})$ is a graded subalgebra satisfying the condition (III).
 We call  $\mathfrak{G}$
\textit{a maximal R-subalgebra (resp. maximal S-subalgebra)} of $L$  if
$\mathfrak{G}$  is maximal reducible (resp.
irreducible) graded and satisfies the condition (III).  All the MGS of type (III) can be split into the disjoint union of  maximal R-subalgebras and maximal S-subalgebras.
\begin{theorem}\label{03} Suppose $L=W$ or $S$.
 \begin{itemize}
 \item[$\mathrm{(1)}$] All maximal R-subalgebras of $L$ are precisely:
$$\{\mathcal{M}(L_{-1},\mathcal{M}_{0}(V))\mid V \in\frak{V}^L\}.$$
\item[$\mathrm{(2)}$]  For any   $V, V' \in\frak{V}^L$,
$$\mathcal{M}(L_{-1},\mathcal{M}_{0}(V))
\stackrel{\mathrm{algebra}}{\cong}\mathcal{M}(L_{-1},\mathcal{M}_{0}(V'))
\Longleftrightarrow V{\cong} V'.$$
\item[$\mathrm{(3)}$]  $L$ has exactly $(m+1)(n+1)-2$
   isomorphism  classes of   maximal R-subalgebras.
 \item[$\mathrm{(4)}$] Suppose $V\in\frak{V}^L$ with $\mathrm{superdim}V=(k,l),$ then\begin{align*}
                                                             &   \mathrm{dim}\mathcal{M}(L_{-1},\mathcal{M}_{0}(V)) \\
                                                            =& \left\{\begin{array}{ll}
                                                                    2^{n-l}p^{m-k}(m+n-k-l)+2^{n}p^m (k+l), & L=W \\
                                                                   2^{n-l}p^{m-k}(m+n)+2^{n}p^m (k+l-1)-k-1, & L=S.
                                                                 \end{array}
    \right.
                                                          \end{align*}
 \end{itemize}
\end{theorem}
Let $L=H$ or $K$ and  put $V^{\bot}=\{u\in L_{-1}\mid \beta(u, V)=0\}$ for $V\in\frak{V}^L$. Recall definitions (\ref{12731eq1})--(\ref{eqsb2}) and (\ref{mgs130407notion}).
\begin{theorem}\label{mgs12710t1}
All    maximal R-subalgebras  of $H$ and $K$   are characterized as follows:

For $H$,
\begin{itemize}
\item[$(1)$]  All maximal  R-subalgebras of $H$  are precisely:
$$\left\{\mathcal{M}(H_{-1}, \mathcal{M}_0(V))\mid V\in\mathcal{V}_{\frak{n}}^H\cup\mathcal{V}_{\frak{i}}^H\right\}.$$
\item[$(2)$] Suppose $V, V'\in \mathcal{V}_{\frak{n}}^H\cup \mathcal{V}_{\frak{i}}^H$.
Then
$$\mathcal{M}(H_{-1}, \mathcal{M}_0(V))\stackrel{\mathrm{algebra}}{\cong}
\mathcal{M}(H_{-1}, \mathcal{M}_0(V'))$$
if and only if one of the following conditions holds.
\begin{itemize}
\item[(i)] $ V{\cong}V'.$
\item[(ii)] $V^{\bot}{\cong}V'$ when  $V$ and $V'$ are both nondegenerate.
  \end{itemize}
\item[$(3)$] $H$ has exactly $\phi(r, n)$  isomorphism  classes of    maximal R-subalgebras,  where
\[
\phi(r, n)=\left\{\begin{array}{ll}
                       2^{-1}(nr+3n+2r-2)+\lfloor\frac{r}2\rfloor (n+1), & n \ \mbox{ is even};\\
 2^{-1}(nr+3n+r-1)+\lfloor\frac{r}2\rfloor (n+1), & n \ \mbox{ is odd}.
                    \end{array}
\right.
\]
\item[$(4)$] Suppose $V\in \mathcal{V}_{\frak{n}}^H\cup \mathcal{V}_{\frak{i}}^H$ with $\beta$-$\mathrm{dim}V=(a, b, c, d),$ then
\[
\dim\mathcal{M}(H_{-1}, \mathcal{M}(V))=\left\{\begin{array}{ll}
                      p^{2a}2^d+p^{2(r-a)}2^{(n-d)}-2, & V\in\mathcal{V}_{\frak{n}}^H;\\
                      p^{m-b}2^{n-c}+(b+c)p^b2^c-1, &  V\in\mathcal{V}_{\frak{i}}^H.
                    \end{array}
\right.
\]
\end{itemize}

For $K$,
\begin{itemize}
\item[$(1')$]
All   maxima R-subalgebras of $K$  are precisely:
$$\left\{\mathcal{M}(K_{-1}, \mathcal{M}_0(V))\mid V\in\mathcal{V}_{\frak{i}}^K\right\}.$$
\item[$(2')$] Suppose $V, V'\in \mathcal{V}_{\frak{i}}^K$.
Then
$$\mathcal{M}(K_{-1}, \mathcal{M}_0(V))\stackrel{\mathrm{algebra}}{\cong}
\mathcal{M}(K_{-1}, \mathcal{M}_0(V'))\Longleftrightarrow V{\cong}V'.$$

\item[$(3')$] $K$ has exactly $\phi(r, n)$  isomorphism  classes of maximal R-subalgebras, where
\[
\phi(r, n)=\left\{\begin{array}{ll}
                       2^{-1}(rn+n+2r-2), & n \ \mbox{ is even};\\
2^{-1}(rn+n+r-1), & n \ \mbox{ is odd}.
                    \end{array}
\right.
\]
\item[$(4')$] Suppose $V\in\mathcal{V}_{\frak{i}}^K$ with  $\beta$-$\mathrm{dim}V=(0, b, c, 0),$ then
\begin{align*}
 \dim\mathcal{M}(K_{-1}, \mathcal{M}_0(V))
= \left\{\begin{array}{ll}
                      p^{b+1}2^c(b+c+1), & J_3 \ \mbox{ is empty };\\
                     p^{b}2^{c}(p^{m-2b}2^{n-2c}+b+c), &  \mbox{ otherwise}.
                    \end{array}
\right.
\end{align*}
\end{itemize}
\end{theorem}
Unfortunately, for maximal S-subalgebras, we have not
obtained a similar description   as for the maximal graded subalgebras of type $(\mathrm{I})$ or $(\mathrm{II})$   as well as for the maximal R-subalgebras. However, the classification of maximal S-subalgebras of $L$ can be reduced to that of the maximal irreducible subalgebras of the classical Lie superalgebras  (see Lemma \ref{dibu}(3)).
\begin{theorem}\label{mgsS-iii} Suppose $L=W, S, H$ or $K$.
All   maximal S-subalgebra of $L$    are characterized as follows:
\item[$\mathrm{(1)}$]  Every maximal S-subalgebra of $L$ is of the form $\mathcal {M}(L_{-1},\mathfrak{G}_{0}),$
where $\mathfrak{G}_{0}$ is a maximal irreducible subalgebra of $L_{0}$.
\item[$\mathrm{(2)}$] Suppose $\mathfrak{G}_{0}$ is a maximal irreducible subalgebra of $L_{0}$.
\begin{itemize}
\item[$\mathrm{(a)}$] For $L=W$,  $\mathcal {M}(W_{-1},\mathfrak{G}_{0})$ is maximal in $W$ if and only if
 $$\mathrm{div}(\mathcal {M}_{1}(W_{-1},\mathfrak{G}_{0}))\neq 0.$$
\item[$\mathrm{(b)}$] For $L=S$ or $H$, $\mathcal {M}(L_{-1},\mathfrak{G}_{0})$ is maximal in $L$ if and only if
$$\mathcal {M}_{1}(L_{-1},\mathfrak{G}_{0})\neq 0.$$
\item[$\mathrm{(c)}$] For $L=K$,  $\mathcal{M}(K_{-1}, \frak{G}_0)$ is a maximal  in $K$ if and only if  there exists $u\in\mathcal{M}_1(K_{-1}, \frak{G}_0)$ satisfying
$$[1, u]\not=0.$$
\end{itemize}
\end{theorem}
Let $L=W, S, H$ or $K$. As in the case of modular Lie algebras \cite{HM}, it is easy to show the following lemmas.
\begin{lemma}\label{lemma113551} Let $\mathfrak{G}_{0}$ be a nontrivial  subalgebra of $L_0$.
  If $\Phi$ is a $\mathbb{Z}$-homogeneous automorphism of $L$. Then
$$\Phi(\mathcal{M}_i(L_{-1},\frak{G}_0))=\mathcal{M}_i(L_{-1}, \Phi(\frak{G}_0))\ \mbox{ for all } i\geq -2.$$
Moreover,
$$\Phi(\mathcal{M}(L_{-1},\frak{G}_0))=\mathcal{M}(L_{-1}, \Phi(\frak{G}_0)).$$
\end{lemma}

\begin{lemma}\label{lem51}
Let $M=L_{-1}+M_{0}+M_{1}+M_{2}+\cdots$ be any MGS of
$L$.  Then
$M_{0}$ is maximal in $L_{0}$ unless   $M_{0}=L_{0}$.
\end{lemma}

\begin{lemma}\label{lem52llm}
If $M$ is an MGS of type $\mathrm{(III)}$ of $L$ then $M_0$ is
maximal in $L_{0}$ and
$M=\mathcal{M}(L_{-1}, M_0)$.
\end{lemma}

\begin{lemma}\label{lem331ok}
 If $\mathfrak{G}_{0}$ is a maximal reducible subalgebra of $L_{0}$ then there exists a $V\in\frak{V}^L$ such that $\mathfrak{G}_{0}=\mathcal
{M}_{0}(V)$ and   $\mathcal{M}_i(L_{-1}, \frak{G}_0)\subset\mathcal{M}_i(V)$
for $i\geq0$.
Conversely,   $\mathcal
{M}_{0}(V)$ is a reducible maximal subalgebra of $L_0$  if  $V\in\frak{V}^L$ when $L=W$ or $S$; if $V\in\mathcal{V}_{\frak{n}}^L\cup\mathcal{V}_{\frak{i}}^L\cup\mathcal{V}_{\frak{d}}^L$ when $L=H$ or $K$.
\end{lemma}
\begin{proof}
Since $\mathfrak{G}_{0}$ is reducible,  $L_{-1}$   has a nontrivial $\mathfrak{G}_{0}$-submodule $V$.
   From definition (\ref{huam}) and the maximality of $\mathfrak{G}_{0}$, we have
$\mathfrak{G}_{0}=\mathcal {M}_{0}(V)$.
From definitions (\ref{huam}) and (\ref{m33}), we have $\mathcal{M}_i(L_{-1}, \frak{G}_0)\subset\mathcal{M}_i(V)$
for $i\geq0$.
The second statement follows immediately from Lemmas \ref{yinlimax-noncontain}(1) and \ref{mgs1273l2}(2).
\end{proof}
\begin{remark}\label{mgsstandard2}
 In view of Lemmas \ref{mgsl3},  \ref{lemma113551} and \ref{lem331ok}, if
  $\mathfrak{G}_{0}$ is a maximal reducible subalgebra of $L_{0}$, we may assume that $V$ is a standard element in $\frak{V}^L$ [see (\ref{eqsb}), \ref{eqsb2})] such that $\frak{G}_0=\mathcal{M}_0(V)$.
\end{remark}

\begin{proposition}\label{Prollm54}  Suppose  $L=W$ or $S$. $\mathcal{M}(L_{-1},\mathfrak{G}_{0})$ is   a maximal R-subalgebra,  if $\mathfrak{G}_{0}$ is a maximal reducible subalgebra
 of $L_{0}$.
\end{proposition}
\begin{proof}
 Let us show that $M=\mathcal{M}(L_{-1},\mathfrak{G}_{0})$  is maximal. Assume that $\overline{M}$ is a maximal graded subalgebra containing  $M$. Clearly, $\overline{M}_{-1}=L_{-1}$. Since $\mathfrak{G}_{0}$ is a maximal subalgebra of $L_{0}$, we have $\overline{M}_{0}=\mathfrak{G}_{0}$ or $L_{0}$.
If $\overline{M}_{0}=\mathfrak{G}_{0}$ then $\overline{M}$ is an MGS of type (III). By Lemma \ref{lem52llm},
$$\overline{M}=\mathcal{M}(L_{-1},\mathfrak{G}_{0})=M $$
 and we are done. Let us consider the remaining case; $\overline{M}_{0}=L_{0}$. Clearly, $\overline{M}$ is an MGS of type (I) and  by  Theorem \ref{thm01},
\begin{equation}\label{eqllm1326}
M_1\subset\overline{M}_{1}=W_1', W_1'', S_1'', \;\mbox{or}\; \{0\}.
\end{equation}
On the other hand, by Lemma \ref{lem331ok}, there exists a  $V\in\frak{V}^L$ such that $\mathfrak{G}_{0}=\mathcal{M}_0(V)$.  Assume that $V$  has a standard basis:
$$
(\partial_{1},\ldots,\partial_{k}\mid \partial_{m+1},\ldots,\partial_{m+l}).
$$
Hence $\mathfrak{G}_{0}=\mathcal{M}_{0}(V)$ has a standard co-basis (\ref{liuliumelikyan311}) in $W_{0}$:
\begin{equation*}\mathcal {A}_{3}=\{x_{i}\partial_{j}\mid i\in
\mathbf{I}(k,l),j\in \overline{\mathbf{I}}(k,l)\}.
\end{equation*}
To reach a contradiction, in view of (\ref{eqllm1326}), it is sufficient to find an element belonging to $M_1$ but not
$W', W''$ for $W$,  but not $S''$ or $\{0\} $ for $S$.  For $L=W$, $x_jx_i\partial_i$ with $i\in \mathbf{I}(k,l)$ and  an arbitrarily chosen $j$ is a desired element. Here we have used the fact that both $|\mathbf{I}(k,l)|\geq 1$ and $|\overline{\mathbf{I}}(k,l)|\geq 1$, since $V\in\frak{V}^W$. For $L = S$, pick distinct $i,j,r$ with $i\in \mathbf{I}(k,l)$ and with $j,r$ chosen arbitrarily. Here note that the general assumption ensures $|\mathbf{I}|\geq 4.$ Then $x_jx_r\partial_i\in S_1$ is a desired candidate for $S$.  The proof is complete.
 \end{proof}

\noindent
\textbf{Proof of Theorem \ref{03}}
 (1) This follows  from Lemmas \ref{lem51}, \ref{lem52llm},  \ref{lem331ok} and Proposition \ref{Prollm54}.

 (2) One implication is obvious. Suppose $\Phi$ is an isomorphism of $\mathcal{M}(L_{-1},\mathcal{M}_{0}(V))$ onto
$\mathcal{M}(L_{-1},\mathcal{M}_{0}(V'))$. Consequently, $\Phi(L_{-1})=L_{-1}$ and
$\Phi(\mathcal{M}_{0}(V))=\mathcal{M}_{0}(V').$ A standard verification shows that
$\Phi(\mathcal{M}_{0}(V))=\mathcal{M}_{0}(\Phi(V))$. By Lemma \ref{yinlimax-noncontain}(2), we have $\Phi(V)=V'$.

(3) This is a direct consequence of (2).

(4) Suppose $V$ is a  standard element in $\frak{V}^L$. Then
$\mathcal{M}(W_{-1},\mathcal{M}_{0}(V))$ has a standard $\mathbb{F}$-basis
\begin{align*}
&\{x^{(\alpha)}x^{u} \partial_{i}\mid \alpha\in
\mathbf{A}(m), u\in \mathbf{B}(n);\;
i\in\mathbf{I}(k,l)\}\\
\cup&\{x^{(\alpha)}x^u
\partial_{i}\mid\alpha_{1}=\cdots=\alpha_{k}=0, u\subset\overline{m+l+1,m+n};\; i\in
\overline{\mathbf{I}}(k,l)\}.
\end{align*}
Thus, we have:
$$\mathrm{dim}\mathcal{M}(W_{-1},\mathcal{M}_{0}(V))=2^{n-l}p^{m-k}(m+n-k-l)+2^{n}p^m (k+l).$$
Note that $\overline{S}=S\oplus\sum_{i\in\mathbf{I}_0}x^{(\pi-(p-1)\varepsilon_i)}x^\omega\partial_i$,
where $\pi=(p-1, \ldots, p-1)\in\mathbb{N}^{m}$ and $\omega=\langle m+1, \ldots, m+n\rangle$.
Then we have: $$\mathrm{dim}\mathcal{M}(S_{-1},\mathcal{M}_{0}(V))=2^{n-l}p^{m-k}(m+n)+2^{n}p^m (k+l-1)-k-1.$$
\qed

We call $\frak{G}_0=\mathcal{M}_0(V)$ is \textit{degenerate}  if $V\in\mathcal{V}_{\frak{i}}^L\cup\mathcal{V}_{\frak{d}}^L$.
\begin{proposition}\label{1278p408}
Let $\frak{G}_0$  be a maximal reducible subalgebra of $H_{0}$ or $K_0$.
 \begin{itemize}
\item[$(1)$] $\mathcal{M}(H_{-1}, \frak{G}_0)$ is maximal in $H$.
\item[$(2)$] $\mathcal{M}(K_{-1}, \frak{G}_0)$ is maximal in $K$ if and only if ${\frak{G}_0}$ is degenerate.
\end{itemize}
\end{proposition}
\begin{proof}
 For any $0\neq h\in L$, $h\not\in\mathcal{M}(L_{-1}, \frak{G}_0)$, put $\overline{M}=\mathrm{alg}(\mathcal{M}(L_{-1}, \frak{G}_0)+\mathbb{F}h)$.
 By the maximality of $\frak{G}_0$, we have $L_0\subset \overline{M}$.
For $H$, choose $k\in I_{0i}$ if $I_{0i}$ is not  empty  where  $i=1, 2$ or 3.
 It follows that  $y^3_k\in\mathcal{M}_1(H_{-1}, \frak{G}_0)$.
For $K$, suppose  ${\frak{G}_0}$ is degenerate. Using the same method as for $H$, we can find $0\not=v_i\in\overline{M}\cap K_{1i}$,  where $i=0, 1$. From Lemmas \ref{dibu}(2) and \ref{mgsl2}, we have $\overline{M}=L$.

It remains to show that  ${\frak{G}_0}$ is degenerate if $\mathcal{M}_1(K_{-1}, \frak{G}_0)$ is maximal.
 Assume on the contrary that $V_{\frak{G}_0}\in\mathcal{V}_{\frak{n}}^K$ is a nondegenerate  irreducible $\frak{G}_0$-module. For any $u\in \mathcal{M}_1(K_{-1}, \frak{G}_0)$,  by Lemmas \ref{mgs1277l1}(2) and \ref{lem331ok}, we may assume that
\[u=f_{-1}z+f_1, \mbox{  where }\ f_{-1}\in V_{\frak{G}_0}\ \mbox{ and }\ f_1\in \widetilde{\mathcal{M}}_1(V_{\frak{G}_0}).\]
 Note that ${f_{-1}}$ is a linear combination of monomials with value 1. Let $f_1=f^1+f^4+f^8$ where $f^i$ is a linear combination of monomials with value $i$, $i=1, 4$ or $ 8$. We  claim that $f_{-1}=0$.
 Indeed, for any $y_i\in K_{-1}$ with value 2, we have
 $$\sigma(i)(-1)^i(D_{\widetilde{i}}(f_1)+D_{\widetilde{i}}(f_{-1})z)+y_if_{-1}=[y_i, u]\in \mathcal{M}_0(V_{\frak{G}_0})=\frak{G}_0,$$
 which implies that $f_{-1}=0$ when $J_3$ is single. Otherwise, the following equation holds:
 \begin{align}
 \sigma(i)(-1)^i D_{\widetilde{i}}(f^4)=-y_if_{-1}.\label{mgs12781837eq1}
 \end{align}
 Then there exists $g_1\in K_1$ with $D_{\widetilde{i}}(g_1)=0$ satisfying
 \begin{align}
\sigma(i)(-1)^i f^4=-y_{\widetilde{i}}y_if_{-1}+g_1.\label{mgs12781856eq2}
\end{align}
 By equations (\ref{mgs12781837eq1}) and (\ref{mgs12781856eq2}), we have
 $D_i(g_1)=(-1)^i2y_{\widetilde{i}}f_{-1}$ which  contradicts $D_{\widetilde{i}}(g_1)=0$ if  $f_{-1}\not=0$.
Consequently, $ \mathcal{M}_1(K_{-1}, \frak{G}_0)\subset K_{10}$. Using  induction on $k$ and the transitivity of $K$, we have
 $ \mathcal{M}_k(K_{-1}, \frak{G}_0)\subset K_{k0}$ for $k>0.$ It follows that
 $\mathcal{M}(K_{-1}, \frak{G}_0)$ is strictly contained in
 $K_{-2}+K_{-1}+K_0+\sum_{i=1}^{2r(p-1)+n}K_{i0}$.  The latter is a maximal graded subalgebra (see Theorem \ref{thm01}(4)). This contradicts the maximality of $\mathcal{M}(K_{-1}, \frak{G}_0)$.
 The proof is complete.
  \end{proof}

\begin{lemma}\label{mgs120906l2} Suppose $L=H$ or $K$.
\begin{itemize}
\item[$(1)$] Suppose $V\in\mathcal{V}_{\frak{d}}^L$. Then $V$ contains a subspace $V'\in\mathcal{V}_{\frak{i}}^L$ such that
$$\mathcal{M}(L_{-1}, \mathcal{M}_0(V))=\mathcal{M}(L_{-1}, \mathcal{M}_0(V')).$$
\item[$(1)$] If $V, V'\in \mathcal{V}_{\frak{n}}^L\cup \mathcal{V}_{\frak{i}}^L$,
then$$\mathcal{M}_0(V)=
\mathcal{M}_0(V') %\left(\mathcal{M}^K_{0}(1, V)=\mathcal{M}^K_{0}(1, V)\right)
$$ if and only if one of the following conditions holds.
\begin{itemize}
\item[$(i)$] $ V=V'.$
\item[$(ii)$] $ V^{\bot}=V'$ when  $V$ and $V'$ are nondegenerate.
\end{itemize}
\end{itemize}
\end{lemma}
\begin{proof}
For (1),  we may assume that $V=\mathrm{span}_{\mathbb{F}}\{y_i\mid i\in J_1\cup J_2\}$. Then $V'=\mathrm{span}_{\mathbb{F}}\{y_i\mid i\in J_2\}$ is desired.
For (2),
by a similar argument as in Lemma \ref{yinlimax-noncontain}(2), we get the desired conclusion.
\end{proof}

\begin{lemma}\label{mgs121113l2}
The following statements hold.
\begin{itemize}
\item[$(1)$]  If $V\in\mathcal{V}_{\frak{n}}^H$, then
$\mathcal{M}(H_{-1}, \mathcal{M}_0(V))=\mathcal{O}_{J_1}\oplus\mathcal{O}_{{J}_3}.$
\item[$(2)$] If $V\in\mathcal{V}_{\frak{i}}^H$, then
 $\mathcal{M}(H_{-1}, \mathcal{M}_0(V))=\mathcal{O}_{J_2\cup J_3}\oplus\mathcal{O}^+_{J_2}\mathcal{Q}_{\bar{J}_2}.$
\end{itemize}
\end{lemma}
\begin{proof}
(1) For $V\in\mathcal{V}_{\frak{n}}^H$,
 a direct computation shows that
\begin{equation*}\label{mgs120924eq2}
\mathcal{M}_i(H_{-1}, \mathcal{M}_0(V))=\mathrm{span}_{\mathbb{F}}\{u\in H_i\mid u\ \mbox{ is a monomial with }\ \nu(u)=1, 2^{i+2}\}.
\end{equation*}

(2) For $V\in\mathcal{V}_{\frak{i}}^H$,   using induction on $i$,
we  obtain that
$\mathcal{M}_i(H_{-1}, \mathcal{M}_0(V))$ is spanned by  monomials in $H$  as follows:
\begin{itemize}
  \item[$(a)$]  $u_1u_2\in H_{i}$, where $u_1$ is a monomial with the variables of  value 0 and  $u_2$ is a monomial with  the variables of  value $2$.
    \item[$(b)$] $y_ju_3\in H_{i}$, where $j\in\bar{J}_2$  and  $u_3$ is a monomial with  the variables of  value $0$.
\end{itemize}
Then,  the conclusions hold.
\end{proof}
For $u\in K$, put $Z(u)=i$ if $(ad1)^{i+1}u=0$ and $(ad1)^{i}u\not=0$.
\begin{lemma}\label{mgs12713l1} Suppose $V\in\mathcal{V}_{\frak{i}}^K$.
For any
%$\mathbb{Z}$-homogenous
element $u\in K$ with $[1, u]\not=0$, $u\in \mathcal{M}(K_{-1}, \mathcal{M}_0(V))$ if and only if
$u$ is a linear combination of elements of the  form $f(z+x)^j+g$,  where
$g\in \mathcal{O}^+_{J_2\cup J_3}\oplus\mathcal{O}^+_{J_2}\mathcal{Q}_{\bar{J}_2}$,
\[
f\in\left\{
\begin{array}{ll}
    \mathcal{O}^+_{J_2}\mathcal{Q}^+_{\bar{J}_2}, & V\in\mathcal{V}_{\frak{i}}^K \mbox { satisfying }  J_3 \ \mbox { is empty}; \\
  \mathcal{O}^+_{J_2\cup J_3}, & V\in\mathcal{V}_{\frak{i}}^K \mbox { satisfying }  J_3 \ \mbox { is not  empty},
\end{array}
\right.
\]
 $x=\sum_{i\in J_2}y_iy_{\widetilde{i}}$ and $0< j<p$.
\end{lemma}
\begin{proof} Let $\frak{G}_0=\mathcal{M}_0(V)$. Notice that $g$, $x$ and $z+x\in\mathcal{M}(K_{-1}, \frak{G}_0)$.
Firstly, for any $i\in \mathbf{I}$, one computes
\begin{align*}
&[y_i, z+x]\in \mathcal{O}^+_{J_2\cup J_3},\\
&f[y_i, z+x]\in\mathcal{O}_{J_2\cup J_3}+\mathcal{O}_{J_2}\mathcal{Q}_{\bar{J}_2},\\
&[y_i, f]\in\left\{
\begin{array}{ll}
    \mathcal{O}^+_{J_2}\mathcal{Q}^+_{\bar{J}_2}, &   V\in\mathcal{V}_{\frak{i}}^K \mbox { satisfying }  J_3 \ \mbox { is empty}; \\
  \mathcal{O}^+_{J_2\cup J_3}, & V\in\mathcal{V}_{\frak{i}}^K \mbox { satisfying }  J_3 \ \mbox { is not empty}.
\end{array}
\right.
\end{align*}
Using  induction on  $\mathrm{zd}(f)$
and $j$, respectively, we have
$$f(z+x),\
(z+x)^{j} \in \mathcal{M}(K_{-1}, \frak{G}_0).$$
Furthermore,
 $f(z+x)^j\in \mathcal{M}(K_{-1}, \frak{G}_0)$.

Conversely,  let us use induction on $Z(u)$.
When $Z(u)=1$,  we consider the following cases.\\

\noindent\textbf{Case 1.} $u\in \mathcal{M}_0(K_{-1}, \frak{G}_0).$
By Lemmas \ref{mgs1277l1}(1) and  \ref{lem331ok}, we may assume that $u=z+u_0$, where $u_0\in \frak{G}_0\cap H_0$, which means that
$u_0\in \mathcal{O}_{J_2\cup J_3}+\mathcal{O}_{J_2}\mathcal{Q}_{\bar{J}_2}$.
Thus, $u=z+x+(u_0-x)$ is desired.\\

\noindent\textbf{Case 2.} $u\in \mathcal{M}_1(K_{-1}, \frak{G}_0).$
From Remark \ref{mgsr2},
%Since torus $T \subset \frak{G}_0$ and the weight space of $K_{11}$ with respect to $T$ is 1-dimension,
we may assume that $u=y_tz+u_1$, where $u_1\in H_1$ and $t\in\mathbf{I}$.
Notice that, when $y_t(z+x)\in  \mathcal{M}(K_{-1}, \frak{G}_0),$
$$u_1-y_tx=u-y_t(z+x)\in \mathcal{M}(K_{-1}, \frak{G}_0)\cap H,$$
which follows that $u=y_t(z+x)+(y_tx-u_1)$ is desired.
Thus, by the necessity of this lemma,
%when $t\in J_2$, $t\in J_3$ or $t\in\bar{J}_2$ with $J_3$ being empty, the conclusion holds.
it is sufficient to consider the case of $t\in\bar{J}_2$ when $J_3$ is not empty.
From Lemmas \ref{mgs1277l1}(3) and  \ref{lem331ok}, we may  assume that
$$u_1=h^{(\frac13, 2, 2)}+h^{(0, \frac13, \frac13)}+h^{(0, \frac13, 2)}+h, $$
where
$$h^{(\alpha, \beta, \gamma)}=\mathrm{span}_{\mathbb{F}}\{y_iy_jy_k\mid \nu({y_i})=\alpha, \nu({y_j})=\beta, \nu({y_k})=\gamma\},$$
%\ a, b, c =0, 2,\mbox{ or }  \frac13$
$h\in \mathcal{M}_1(K_{-1}, \frak{G}_0)\cap H$.
 For any $y_l\in K_{-1}$, we have:
$$\sigma(l)(-1)^l(D_{\widetilde{l}}(y_t)z+D_{\widetilde{l}}u_1)+y_ly_t=[y_l, u]\in \frak{G}_0,$$ which means that
\begin{equation}\label{mgs127131825eq2}
\sigma(l)(-1)^lD_{\widetilde{l}}u_1+y_ly_t\in \frak{G}_0.
\end{equation}
When $\nu(y_l)=2$, from equation (\ref{mgs127131825eq2}) we have:
$$\sigma(l)(-1)^lD_{\widetilde{l}}h^{(\frac13, 2, 2)}+y_ly_t\in \frak{G}_0,$$
which  is
a linear combination  of elements with value $\frac23$. It follows
that
\begin{equation*}\label{mgs127131901eq4}
\sigma(l)(-1)^lD_{\widetilde{l}}h^{(\frac13, 2, 2)}+y_ly_t=0.
\end{equation*}
 When $J_3$ is single,  we have $y_ly_t=0$, a contradiction. When $J_3$ is not single,
there exist distinct $k, \widetilde{k}\in J_3$  such that
\begin{equation}\label{mgs127131901eq5}
h^{(\frac13, 2, 2)}=-\sigma(k)(-1)^ky_{\widetilde{k}}y_ky_t+h',
\end{equation}
where $D_{\widetilde{k}}h'=0$ and
%Since there are at least three elements in $J_3$, there exists $k\not=\widetilde{k}\in J_3$ satisfying equation (\ref{mgs127131901eq4}), i.e.,
$$\sigma(\widetilde{k})(-1)^kD_{k}(h^{(\frac13, 2, 2)})+y_{\widetilde{k}}y_t=0.$$
From equation (\ref{mgs127131901eq5}), we have
$$D_k(h')=D_k(h^{(\frac13, 2, 2)})+\sigma(k)y_{\widetilde{k}}y_t=2\sigma(k)y_{\widetilde{k}}y_t,$$
which contradicts $D_{\widetilde{k}}h'=0$. Thus, an element of the  form $y_tz+u_1$, $t\in\bar{J}_2$ is not in $\mathcal{M}_1(K_{-1}, \frak{G}_0)$ when $J_3$ is not empty.\\

\noindent
 \textbf{Case 3.}  $u\in \mathcal{M}_i(K_{-1}, \frak{G}_0)$ for $i>1$. We may assume that
 $$u=g_{i-2}z+g_i,\
 g_{j}\in \overline{H}_{j}, \ j=i-2, i.$$
  Note that the elements of the  form $h_2z+h$ are not in $\mathcal{M}(K_{-1}, \frak{G}_0)$,
 where $h_2$ is a linear combination of monomials with value $\frac19$.
By induction on $ i$, we obtain that   $g_{i-2}$ is in $ \mathcal{O}_{J_2}\mathcal{Q}_{\bar{J}_2}$ if $J_3$ is empty;  in $\mathcal{O}_{J_2\cup J_3}$, otherwise.
Thus, $g_{i-2}(z+x)\in \mathcal{M}_i(K_{-1}, \frak{G}_0)$. Moreover,
$g_i-g_{i-2}x\in\mathcal{M}_i(K_{-1}, \frak{G}_0)\cap\overline{H}$.
Then $u=g_{i-2}(z+x)+(g_i-g_{i-2}x)$ is desired.

 When $Z(u)=k>1$, suppose
$$u=u_kz^k+u_{k-1}z^{k-1}+\cdots+u_1z+u_0, \  u_j\in\overline{H}, \ j=0,\ldots, k.$$
Obviously,
$$u_kz+u_{k-1}=2^{(1-k)}(\mathrm{ad}1)^{k-1}(u)\in \mathcal{M}(K_{-1}, \frak{G}_0).$$
 Thus,
$u_k$ is in $  \mathcal{O}^+_{J_2}\mathcal{Q}^+_{\bar{J}_2}$ when $J_3$ is empty; in $\mathcal{O}^+_{J_2\cup J_3}$, otherwise.
Consequently, $u_k(z+x)^k\in \mathcal{M}(K_{-1}, \frak{G}_0).$
Thus,
$$v=u-u_k(z+x)^k\in \mathcal{M}(K_{-1}, \frak{G}_0)$$
and $Z(v)<k$. By the inductive hypothesis, $v$ is a linear combination of the desired form. So is $u$. The proof is complete.
\end{proof}

%%%%%%%%%%%%%%%%%%%%%%%%%%%%%%%%%%%%%%%%%%%%%%%%%%%%%%%%%%%%%%%%%%%%%%%%%%%%%%%%%%%%%%%%%%%%%%%%%%%%%%%%%%%%%
%%%%%%%%%%%%%%%%%%%%%%%%%%%%%%%%%%%%%%%%%%%%%%%%%%%%%%%%%%%%%%%%%%%%%%%%%%%%%%%%%%%%%%%%%%%%%%%%%%%%%%%%
%%%%%%%%%%%%%%%%%%%%%%%%%%%%%%%%%%%%%%%%%%%%%%%%%%%%%%%%%%%%%%%%%%%%%%%%%%%%%%%%%%%%%%%%%%%%%%%%%%%%%%%%
\noindent
\textbf{Proof  of Theorem \ref{mgs12710t1}}.
For (2) and ($2'$), sufficiency  is obvious.
For necessity, suppose $\Phi$ is an isomorphism of $\mathcal{M}(L_{-1}, \mathcal{M}_0(V))$ onto $\mathcal{M}(L_{-1}, \mathcal{M}_0(V'))$.
 Then, $\Phi(L_{-1})=L_{-1}$ and $\Phi(\mathcal{M}_0(V))=\mathcal{M}_0(V')$, which implies that
 $\dim\mathcal{M}_0(V)=\dim\mathcal{M}_0(V')$. It follows that
 $V$ and $V'$ are both nondegenerate or are both isotropic.  Notice  that
 $\Phi(\mathcal{M}_0(V))\subset\mathcal{M}_0(\Phi(V))$. For the maximality of $\mathcal{M}_0(V')$, we have
 $\mathcal{M}_0(V')=\mathcal{M}_0(\Phi(V)).$
By virtue of Lemma  \ref{mgs120906l2}(2), we have $V'=\Phi(V)$ or  $V'=\Phi(V)^{\bot}$  when $V'$ and $\Phi(V)$
are both nondegenerate.
Thus, we have $\dim V=\dim V'$ or $\dim V=m+n-\dim V'$.
% when $V'$ and $\Phi(V)$ are nondegenerate.
We can obtain the desired conclusions by a direct computation. (3) and $(3')$ are direct consequences of (2) and $(2')$.

The remaining statements hold from
Lemmas \ref{lem51},  \ref{lem52llm},  \ref{mgs120906l2},   \ref{mgs121113l2}, \ref{mgs12713l1}  and Proposition \ref{1278p408}.\qed

Finally, we consider the maximal  S-subalgebras  of $L$,  where $L=W, S, H$ or $K$.
As in the case of modular Lie algebras \cite{HM}, it easy to show the following:
\begin{lemma}
\label{mgs12714l1} Suppose  $\frak{G}_0$ is a maximal irreducible
subalgebra of $L_0$.
The subalgebra  $\mathcal{M}(L_{-1}, \frak{G}_0)$ is not maximal in $L$ if $\mathcal{M}_1(L_{-1}, \frak{G}_0)=0$.
\end{lemma}
%%%%%%%%%%%%%%%%%%%%%%%%%%%%%%%%%%%%%%%%%%%%%%%%%%%%%%%%%%%%%%%%%%%%%%%%%%%%%%%%%%%%%%%%%%%%%%%%%%%%%%%%
\noindent
\textbf{Proof  of Theorem \ref{mgsS-iii}}.
(1) This is nothing but Lemma \ref{lem52llm}.

(2) Let $\mathfrak{G}_0$ be a maximal irreducible subalgebra of $L_0.$

(a) Suppose $L=W$ and $\mathcal {M}(W_{-1},\mathfrak{G}_{0})$ is maximal in
$W$. Assume on the contrary  that $\mathrm{div}(\mathcal{M}_1(W_{-1},\mathfrak{G}_{0}))=0$. By induction on $i$,  one has $\mathcal{M}_{i}(W_{-1},\mathfrak{G}_{0})\subset W'_{i}$ for all $i\geq1.$
Since $\mathfrak{G}_{0}$ is a nontrivial subalgebra of $W_{0}$,  we have
\begin{equation*}\label{41}
\mathcal{M}(W_{-1},\mathfrak{G}_{0})\subsetneq W_{-1}+W_{0}+W'_{1}+W'_{2}+\cdots
\end{equation*}
By Theorem \ref{thm01}, the latter is  an MGS
of $W$. This  contradicts the maximality of $\mathcal{M}(W_{-1},\mathfrak{G}_{0})$.

Conversely, to show the maximality of $\mathcal{M}(W_{-1},\mathfrak{G}_{0})$, assume that $M$ is an MGS strictly containing  $\mathcal{M}(W_{-1},\mathfrak{G}_{0})$. By definition (\ref{m33}), it must be that $M_0\supsetneq\mathfrak{G}_0$ and therefore, $M_0=W_0$ by the maximality of $\mathfrak{G}_0$. Thus $M$  is an MGS of type (I) and thereby
\begin{equation}\label{eqllm1825t}
M_1=W_1'\; \mbox{or}\; W_1''.
\end{equation}
  Note that
$W''_{1}$ is an irreducible $\mathfrak{G}_{0}$-module,  which follows from the irreducibility of $\mathfrak{G}_{0}$ and a simple fact that, as $W_0$-modules,
$$W''_{1}\cong(W_{-1})^{*}.$$
By our assumption, there is a
$D\in\mathcal {M}_{1}(W_{-1},\mathfrak{G}_{0})\subset M_1$ with $\mathrm{div}D$ $\neq 0$. Assert that
$D\notin W''_{1}$.
Assuming on the contrary, by the irreducibility of $W''_{1}$, we have $W''_{1}\subset\mathcal
{M}_{1}(W_{-1},\mathfrak{G}_{0})$ and thereby
\begin{equation*}
W_{0}=\mathrm{alg}([W_{-1},W''_{1}])\subset \mathrm{alg}([W_{-1},\mathcal{M}_{1}(W_{-1},\mathfrak{G}_{0})])\subset \mathfrak{G}_{0}.
\end{equation*}
This contradicts the assumption that $\mathfrak{G}_{0}$ is a nontrivial subalgebra of $W_{0}$ and hence the assertion holds. This proves that $D$ belongs to neither $W_1'$ nor $W_1''$, contradicting (\ref{eqllm1825t}).

(b) For  $S$,  from Lemma \ref{mgs12714l1}, one implication is obvious.   As in (a), we have $\mathcal{M}(S_{-1},\mathfrak{G}_{0})$ is maximal when $\mathcal {M}_{1}(S_{-1},\mathfrak{G}_{0})\neq
0$.
For $H$,  the conclusion follows from  Lemmas \ref{mgsl2}(1) and  \ref{mgs12714l1}.

(c) Suppose $L=K$.
 Assume on the contrary that  $[1, u]=0$ for every $u\in\mathcal{M}_1(K_{-1}, \frak{G}_0)$.  Then,
$$ \mathcal{M}_1(K_{-1}, \frak{G}_0)\subset K_{10}.$$
As in the proof of Proposition \ref{1278p408}(2), we have  $\mathcal{M}(K_{-1}, \frak{G}_0)$ is not maximal.

Conversely,  suppose  $u=u_0+u_1$ where $u_i\in K_{1i}$, $i=0, 1$ and $u_1\not=0$. We claim that $u_0\not=0$.
Indeed, by a direct computation, $[K_{-1}, K_{11}]=K_0$ holds. Assuming on the contrary that  $u_0=0$, we have $u_1\in \mathcal{M}(K_{-1}, \frak{G}_0)$. $K_{-1}$  is an  irreducible $\frak{G}_0$-module, and so is $K_{11}$.
Moreover, $K_{11}\subset\mathcal{M}(K_{-1}, \frak{G}_0)$.
Thus,
$$[K_{-1}, K_{11}]\subset[K_{-1}, \mathcal{M}(K_{-1}, \frak{G}_0)]\subset \frak{G}_0\subsetneq K_0,$$
 which contradicts $[K_{-1}, K_{11}]=K_0$.

Put $h\in K$,  $h\not\in \mathcal{M}(K_{-1}, \frak{G}_0)$ and $\overline{M}=\mathrm{alg}{(\mathcal{M}(K_{-1}, \frak{G}_0)+\mathbb{F}h)}$.
By  definition (\ref{m33}) and the maximality of $\frak{G}_0$, we have $K_0\subset\overline{ M}$.
Since the torus $T$  is in $\overline{M}$, from Remark \ref{mgsr2}, there exists
$$v=v_0+y_iz\in\mathrm{alg}{(\mathbb{F}u+T)}\subset\overline{M}, \ i\in\mathbf{I},\ v_0\in K_{10}.$$
 For $K_0\subset\overline{ M}$, without loss of generality, we may assume that $ i\in\mathbf{I}_0.$
If $v_0=0$, we have $y_iz\in\overline{M}$. For $u_0\not=0$, the conclusion holds.  Otherwise, we claim that
 there exists a nonzero element in $ \overline{M}\cap K_{10}$.
Indeed,  it is sufficient to consider the following cases.\\

\noindent
 \textbf{Case 1.} $D_{\widetilde{i}}(v_0)\not=0$.  Note that
 $0\not=[y^{(2\varepsilon_i)}, v]\in \overline{M}\cap K_{10}.$ \\

 \noindent
  \textbf{Case 2.}  $D_{\widetilde{i}}(v_0)=0$ and there exists $t\in\mathbf{I}$, $t\not=i$,  $\widetilde{i}$, such that $D_{\widetilde{t}}(v_0)\not=0$. Note that
 $0\not=[y_ty_i, v]\in \overline{M}\cap K_{10}.$
\\

\noindent
 \textbf{Case 3.} $D_{{t}}(v_0)=0$, for all $t\in\mathbf{I}\backslash\{i\}$. Then $v_0=y^{(3\varepsilon_i)}.$
Note that
$$y^{(3\varepsilon_i)}=(2\sigma(\widetilde{i}))^{-1}([y_iy_{\widetilde{i}}, v]-\sigma(\widetilde{i})v)\in \overline{M}\cap K_{10}.$$
By Lemmas \ref{dibu}(2), \ref{mgsl2}(2) and  $u_i\not=0$  for  $i=0, 1$, the conclusion follows.\qed

\end{document}